\documentclass[11pt]{amsart}
\usepackage{amsmath}
\usepackage[dvips]{epsfig}
\usepackage{xcolor}
\theoremstyle{plain}
\newtheorem{theorem}{Theorem}[section]
\newtheorem{lemma}[theorem]{Lemma}
\newtheorem{prop}[theorem]{Proposition}
\newtheorem{cor}[theorem]{Corollary}

\usepackage{enumitem,blkarray}
\theoremstyle{definition}
\newtheorem{remark}[theorem]{Remark}

\numberwithin{equation}{section}
\usepackage{hyperref}

\title{A Minkowski-type inequality in the AdS-Melvin space}
\author[D. Xia]{Daniel Xia}
\address{Ridge High School, New Jersey, NJ 07920}
\email{danielxia03@gmail.com}
\author[P.-K. Hung]{Pei-Ken Hung}
\address{School of Mathematics, University of Minnesota, 
Minneapolis, MN 55455}
 \email{pkhung@umn.edu}
\date{November 2021}

\begin{document}

\begin{abstract}
The AdS-Melvin spacetime was introduced by Astorino and models the AdS soliton with electromagnetic charge. It is a static spacetime with a time-symmetric Cauchy hypersurface, which we refer to as the AdS-Melvin space. In this paper, we study a sharp Minkowski-type inequality for surfaces embedded in the AdS-Melvin space. We first prove the inequality for special cases in which the surface enjoys axisymmetry or is a small perturbation of a coordinate torus. We then use a weighted normal flow to show that the inequality holds for general surfaces. %We show that under a suitable flow, the objective function is nonincreasing, and establish a series of estimates to show that the limit of the objective function is zero.
\end{abstract}

\maketitle

\section{Introduction}
The classical Minkowski inequality for a closed convex hypersurface $\Sigma$ in $\mathbb{R}^n$ reads
\begin{equation}\label{minko}
\int_{\Sigma}H\, dA\geq |\mathbb{S}^{n-1}|^{\frac{1}{n-1}}|\Sigma|^{\frac{n-2}{n-1}}.
\end{equation}Here $H$ stands for the mean curvature of $\Sigma$. Moreover, equality holds if and only if $\Sigma$ a round sphere. Applying the inverse mean curvature flow, Guan and Li  \cite{Guan} showed that the convexity condition for \eqref{minko} can be relaxed to mean convexity and star-shapedness. The star-shapedness can be further weakened to an outward minimizing condition using the weak inverse mean curvature flow of Huisken and Ilmanen \cite{HuiskenIlmanen}. It is still an open problem whether \eqref{minko} holds for general mean convex hypersurfaces.

In \cite{Penrose}, Penrose gave an interpretation of \eqref{minko} from the view of general relativity and cosmic censorship. The total mean curvature is viewed as the amount of mass needed to turn $\Sigma$ into a black hole and \eqref{minko} can be (heuristically) derived from the Penrose inequality. We refer readers to  \cite{Mars} for a detailed account of the Penrose ineqaulity and the connection to \eqref{minko}.\\ 

Based on the black hole collapse interpretation, \eqref{minko} has been generalized to other ambient spaces. In  \cite{BrendleHuanWang}, Brendle, Wang and the second author extended this inequality to convex, star-shaped hypersurfaces $\Sigma$ in the Anti-de-Sitter Schwarzschild space. Let $\Sigma_{H}$ be the horizon and $\phi$ be the static potential in the Anti-de-Sitter Schwarzschild space. The inequality reads
\begin{equation}
\int_{\Sigma}\phi H\, dA-n(n-1)\int_{\Omega}\phi\, dV\geq (n-1)|\mathbb{S}^{n-1}|^{\frac{1}{n-1}}\left(|\Sigma|^{\frac{n-2}{n-1}}-|\Sigma_H|^{\frac{n-2}{n-1}}\right),
\end{equation}where $\Omega$ is the bounded region with boundary $\partial\Omega=\Sigma\cup\Sigma_H$. Moreover, equality holds if and only if $\Sigma$ is a coordinate sphere. For a bounded region with outward minimizing boundary
$\Sigma$ in the Schwarzschild space with mass $m$, Wei  \cite{YWei} proved that
\begin{equation}\label{eq1.3}
\frac{1}{(n-1)|\mathbb{S}^{n-1}|}\int_{\Sigma}\phi H\, dA\geq \left(\frac{|\Sigma|}{|\mathbb{S}^{n-1}|}\right)^{\frac{n-2}{n-1}}-2m,
\end{equation}with rigidity for a coordinate sphere. Recently, McCormick \cite{McCormickM} showed that the inequality \eqref{eq1.3} holds for asymptotically flat static spacetimes of dimension $3\leq n\leq 7$. Note that all of these inequalities are for static manifolds with spherical infinity.\\

In 1999, Horowitz and Myers \cite{HM} discovered an exact solution to the Einstein equations with a negative cosmological constant, which they called an AdS soliton. It is static and asymptotically locally Anti-de Sitter. The static slice, referred to as \emph{the Horowitz-Myers geon} \cite{Woolgar}, is asymptotically locally hyperbolic with a toroidal conformal infinity.

Horowitz and Myers postulated that the Horowitz-Myers geon is the ground state for a conjectural non-supersymmetric AdS/CFT correspondence \cite{HM}. More precisely, let $(M,g)$ be an asymptotically hyperbolic manifold with toroidal conformal infinity and scalar curvature greater than or equal to $-n(n+1)$. Then the total mass is at least equal to the mass of the Horowitz-Myers geon. Moreover, the rigidity holds if and only if $(M,g)$ is isometric to the Horowitz-Myers geon. Progress on the Horowitz-Myers conjecture has been very limited. For small perturbations of Horowitz-Myers the conjecture was proven by Constable and Myers \cite{CostableMyers} and Woolgar proved the rigidity of this conjecture in 3 dimensions  \cite{Woolgar}. Recently, a Minkowski-type inequality was proved on the Horowitz-Myers geon by Alaee and the second author \cite{Alaee-Hung}. In this paper, we consider its generalization to the AdS soliton coupled with a Maxwell field. Such a spacetime is called called \emph{the AdS-Melvin spacetime} and we give a short description in the next paragraph.  \\

In \cite{Astorino}, a generalization of AdS soliton is introduced and is referred to as the AdS-Melvin spacetime. The AdS-Melvin spacetime is a static solution to the Einstein-Maxwell equations with a negative cosmological constant. The spacetime metric and the potential of the Maxwell field have the explicit form \cite{KT}
\begin{align*}
&^{(4)}{g}_{\mu\nu}dx^\mu dx^\nu=-r^2dt^2+r^{-2}F(r)^{-1}dr^2+r^2 dx^2+r^2F(r)dy^2,\\  &F(r)=1-r^{-3}-br^{-4},\ A_\mu dx^\mu=-\frac{b^{1/2}}{r}dy.
\end{align*}            
Here $b>0$ is a positive constant.  The ranges of the coordinates are \begin{align}
t\in \mathbb{R},\ r\in (r_{s},\infty),\ x\in [0,P_x],\ y\in [0,P_y].
\end{align}
Here $r_{s}$ is the largest positive root of $F(r)$. We identify $x=x+P_x$ and $y=y+P_y$. In particular, the $(x,y)$ coordinates parametrize a torus. The period $P_x$ is an arbitrary positive number and the period $P_y$ is chosen such that the metric can be extended smoothly to $r=r_{s}$. %Performing a change of variables $r=r_{s}+\alpha \rho^2,\ y=\beta\phi$, we deduce that near $r=r_{s}$,
%\begin{align*}
%    r^{-2}F(r)^{-1}dr^2+r^2F(r)dy^2\approx \frac{4\alpha}{r_{s}^2F'(r_{s})}d\rho^2+\alpha\beta^2 r_{s}^2F'(r_{s}) (\rho^2 d\phi^2).
%\end{align*} 
%Choosing $\alpha= 4^{-1}r^2_sF'(r_{s})$ and $\beta=\frac{2}{r_{s}^2F'(r_{s})}$, 
It is straightforward to check that $$P_y=\frac{4\pi}{r_{s}^2F'(r_{s})}.$$
When we send $b$ to zero, the AdS-Melvin spacetime reduces to the AdS soliton introduced by Horowitz and Myers.\\

In this paper, we mainly consider the time-symmetric Cauchy hypersurface $M=\{t=0\}$ in the AdS-Melvin spacetime and we refer to it as the \textit{AdS-Melvin space}. Topologically, $M$ is homeomorphic to a solid torus. The induced metric on $M$ is given by
\begin{align}\label{metric}
    \bar{g}_{ab}dx^adx^b=r^{-2}F(r)^{-1}dr^2+r^2 dx^2+r^2F(r)dy^2.
\end{align}
The metric in \eqref{metric} is complete and is asymptotically locally hyperbolic. Actually, if we replace $F(r)$ in \eqref{metric} by $1$, $g$ becomes the hyperbolic metric under the horosphere coordinates. 

Let $\Sigma$ be a surface embedded in $M$ which is a graph over the torus. In other words, there is a smooth function $s(x,y)$ such that $\Sigma=\{r=s(x,y)\}$. Let $g$ be the induced metric on $\Sigma$ and $H$ be the mean curvature with respect to the outward normal vector. Denote by $\Omega$ the region enclosed by $\Sigma$. Define the quantity $Q(\Sigma)$ by
\begin{equation}\label{def:Q}
Q(\Sigma):=\int_{\Sigma} H r\, dA-6\int_{\Omega} r\, dV.
\end{equation}  We are mainly concerned with the following Minkowski-type inequality:
\begin{align}\label{main} 
    Q(\Sigma) \geq P_xP_y(2r_s^3-2^{-1}).
\end{align} 
In this paper, we prove that \eqref{main} holds under various assumptions on $\Sigma$. Moreover, under those assumptions, we show that equality holds only if $\Sigma$ is a coordinate torus given by $\{r\equiv \textup{const}\}$.

To start, we consider a surface $\Sigma$ which is symmetric along the $x$ or $y$ direction. For these types of surfaces, \eqref{main} holds true.
\begin{theorem}\label{thm:symmetry}
Let $\Sigma$ be a surface embedded in $M$ which is a graph over the torus. Assume that $\Sigma$ is symmetric along the $x$ or $y$ direction. Then \eqref{main} holds. Moreover, equality is achieved if and only if $\Sigma$ is a coordinate torus.
\end{theorem} 
Our next result considers the validity of \eqref{main} when $\Sigma$ is a perturbation of a coordinate torus. 
\begin{theorem}
\label{thm:perturbation}
Fix $r_0>r_s$ and a smooth function $\phi(x,y)$ on the torus. For $\varepsilon$ small, consider \begin{equation}
\label{perturbationdef_intro}
\tilde{s}_\varepsilon(x,y) = r_0 + \varepsilon \phi(x,y) 
\end{equation} 
and the graph $\tilde{\Sigma}_\epsilon$ given by
\begin{align*}
\tilde{\Sigma}_\varepsilon=\{ (\tilde{s}_\varepsilon(x,y),x,y)\, |\, (x,y)\in T^2 \}.
\end{align*}
Then we have
\begin{equation}
\label{firstorder_intro}
\frac{d}{d\varepsilon}Q(\tilde{\Sigma}_{\varepsilon})\bigg|_{\varepsilon=0} =0, 
\end{equation}
and
\begin{equation}
\label{secondorder_intro}
\frac{d^2}{d\varepsilon^2}Q(\tilde{\Sigma}_{\varepsilon})\bigg|_{\varepsilon=0} \geq 0.
\end{equation}
Moreover, equality holds in $\eqref{secondorder_intro}$ if and only if $\phi$ is a constant function.
\end{theorem}

After establishing \eqref{main} for the special cases in Theorem~\ref{thm:symmetry} and Theorem~\ref{thm:perturbation}, we use a normal flow to investigate more general surfaces. Concretely, given a surface $\Sigma$ in $M$, we consider the normal flow $\varphi:[0,T_0)\times T^2\to M$ with initial data $\Sigma$ that satisfies
\begin{align}\label{equ:trueflow_intro}
\frac{\partial  \varphi}{\partial t } = r^{-1}\nu.
\end{align}
The speed $r^{-1}$ is chosen such that the quantity $Q$ is monotonically nonincreasing. See Proposition~\ref{pro:monotonicity} for more details. We show that \eqref{main} holds provided we can extend $T_0$ to infinity.  
\begin{theorem}
\label{theorem:main}
Let $\Sigma$ be an embedded surface in $M$ which is a graph over the torus. Assume that the flow \eqref{equ:trueflow_intro} exists for all time. Then \eqref{main} holds. Moreover, equality is achieved if and only if $\Sigma$ is a coordinate torus.
\end{theorem}

\subsection*{Notation}
  We explain here the notation we use in this paper. We use $x^a,x^b$,  $a,b=1,2,3$ to denote a coordinate of the AdS-Melvin space $M$. We write $\bar{g}_{ab}$ for the AdS-Melvin metric on $M$ and write $\bar{\nabla}$ for the corresponding Levi-Civita connection. The volume form is denoted as $dV$. We write $F(r)$ for the function
  \begin{equation}\label{def:F}
  F(r)=1-r^{-3}-br^{-4}.
  \end{equation}
  We often omit the arguments in $F(r)$ and $\frac{d F}{dr}(r)$ and denote them by $F$ and $F'$  respectively. For any $\lambda \in \mathbb{R}$, we write $F^a$ for $(F(r))^\lambda$.
  
  For an embedded suface $\Sigma$ in $M$, we use $x^i,x^j$, $i,j=1,2$ to denote a coordinate on $\Sigma.$ The induced metric of $\Sigma$ is denoted by $g_{ij}$ and we write $\nabla$ for the Levi-Civita connection. The unit outward normal vector of $\Sigma$ is denoted by $\nu$ and $H$ is the mean curvature with respect to $\nu$. The area form is denoted by $dA$. In the case that $\Sigma$ is a graph over the torus as $\Sigma=\{r=s(x,y)\}$, we write $F$, $F^\lambda$ and $F'$ as abbreviations of $F(s(x,y))$, $(F(s(x,y)))^\lambda$ and $\frac{dF}{dr}(s(x,y))$ respectively.  \\   

This paper is organized as follows. In Section 2, we compute the geometric quantities appearing in \eqref{main} in terms of the height function $s(x,y)$. In Section 3, we prove Theorem~\ref{thm:symmetry} and Theorem~\ref{thm:perturbation} using formulas derived in Section 2. In Section 4, we establish the monotonicity of $Q(\Sigma_t)$ along the flow \eqref{equ:trueflow_intro}. The limit of $Q$ is computed in Section 5, leading to the proof of Theorem~\ref{theorem:main}.

\section{Geometry of graphs}
In this section we record basic geometric quantities of an embedded surface $\Sigma\subset M$ that is a graph over the torus. We start with the Christoffel symbols of $(M,\bar{g})$. Recall that $M$ is covered by a coordinate system $(r,x,y)\in (r_{s},\infty)\times T^2$ and the metric $\bar{g}$ is given by \eqref{metric}.
\begin{lemma}\label{lem:christoffel}
The nonzero Christoffel symbols are as follows:
\begin{equation}\label{equ:Gamma}
\begin{split}
   & \bar{\Gamma}_{rr}^r = \frac 12 r^2F   \big(r^{-2}F^{-1} \big)', \ \bar{\Gamma}_{xx}^r = -r^3F , \\
    &\bar{\Gamma}_{yy}^r = -\frac 12 r^2F   \big(r^2F \big)',\     \bar{\Gamma}_{rx}^x  = \bar{\Gamma}_{xr}^x = r^{-1}, \\
    &\bar{\Gamma}_{ry}^y = \bar{\Gamma}_{yr}^y = \frac{1}{2}r^{-2}F^{-1}\big(r^2F\big)'. 
\end{split}
\end{equation}
\end{lemma}

Next, we consider a surface which is a graph over the torus. Let $s(x,y)$ be a smooth function defined on $T^2$. Let 
\begin{align}\label{equ:Sigma}
\Sigma=\{(s(x,y),x,y )\,|\, (x,y)\in T^2 \} 
\end{align}  
with the parameterization $\varphi(x,y) = (s(x,y),x,y).$ We compute below the induced metric $g$ on $\Sigma$. From
\begin{equation}\label{equ:tangent}
\frac{\partial\varphi}{\partial x}  = \frac{\partial s}{\partial x}\frac{\partial}{\partial r} + \frac{\partial}{\partial x}, \ \frac{\partial\varphi}{\partial y}  = \frac{\partial s}{\partial y}\frac{\partial}{\partial r} + \frac{\partial}{\partial y},
\end{equation}
we have
\begin{equation}\label{equ:inducedmetric}
\begin{split}
g_{xx} &= r^2+r^{-2}F^{-1}\left(\frac{\partial s}{\partial x}\right)^2, \ g_{xy}= r^{-2}F^{-1} \frac{\partial s}{\partial x}\frac{\partial s}{\partial y},\\
g_{yy} &= r^2F+r^{-2}F^{-1}\left(\frac{\partial s}{\partial y}\right)^2.
\end{split}
\end{equation} 
Under these coordinates, we have
\begin{equation}\label{equ:detg}
 \det g =r^6F^2N^2,
\end{equation}
where $N$ is defined by
\begin{equation}
\label{equ:defN}
 N^2 = r^{-2}F^{-1}\left(1+r^{-4}F^{-1}\left(\frac{\partial s}{\partial x}\right)^2 +r^{-4}F^{-2}\left(\frac{\partial s}{\partial y}\right)^2 \right). 
\end{equation}
The inverse metric is then given by
\begin{equation}\label{equ:inversemetric}
\begin{split}
r^6F^2N^2 g^{xx}&=r^2F+r^{-2}F^{-1}\left(\frac{\partial s}{\partial y}\right)^2, \ r^6F^2N^2 g^{xy}= -r^{-2}F^{-1} \frac{\partial s}{\partial x}\frac{\partial s}{\partial y},\\
r^6F^2N^2 g^{yy}&= r^2+r^{-2}F^{-1}\left(\frac{\partial s}{\partial x}\right)^2.
\end{split}
\end{equation}
The area form of $\Sigma$ and volume form of $M$ are computed to be 
\begin{equation}
\label{volforms}
dA  = r^3FN\,   dx  dy, \  {dV} = r\,   dr  dx  dy.
\end{equation}
We denote by $\nu$ the unit outward normal vector of $\Sigma$. From \eqref{equ:tangent} we have
\begin{equation}\label{equ:normal}
\nu=N^{-1}\left( \frac{\partial}{\partial r}-r^{-4}F^{-1}\frac{\partial s}{\partial x} \frac{\partial}{\partial x}-r^{-4}F^{-2}\frac{\partial s}{\partial y} \frac{\partial}{\partial y} \right).
\end{equation}

Consider the function $z (x,y)$ defined through 
\begin{align}\label{def:z}
z^2 = 1+r^{-4}F^{-1}\left( \frac{\partial s }{\partial x} \right)^2+r^{-4}F^{-2}\left(\frac{\partial s }{\partial y} \right)^2. 
\end{align}
In view of \eqref{equ:defN}, the relation between $z$ and $N$ is given by
\begin{align}
N^2=r^{-2}F^{-1}z^2.
\end{align}
Recall that $\nabla$ is the Levi-Civita connection on $\Sigma$. In the lemma below, we compute $|\nabla r|^2$ and $|\nabla y|^2$. These formulas will be used in expressing $\Delta r$ in terms of the mean curvature $H$.
\begin{lemma}
We have
\begin{equation}\label{equ:gradr}
|\nabla r|^2=r^2 F(1-z^{-2}),
\end{equation}
and
\begin{equation}\label{equ:grady}
|\nabla y|^2=r^{-2}F^{-1}-r^{-6}F^{-3}z^{-2}\left(\frac{\partial s}{\partial y} \right)^2.
\end{equation}
\end{lemma}
\begin{proof}
In terms of $z$, the inverse metric can be expressed as
\begin{equation}\label{equ:inversemetric_z}
\begin{split}
r^4Fz^2 g^{xx}&=r^2F+r^{-2}F^{-1}\left(\frac{\partial s}{\partial y}\right)^2, \ r^4Fz^2 g^{xy}= -r^{-2}F^{-1} \frac{\partial s}{\partial x}\frac{\partial s}{\partial y},\\
r^4Fz^2 g^{yy}&= r^2+r^{-2}F^{-1}\left(\frac{\partial s}{\partial x}\right)^2.
\end{split}
\end{equation}
Through a direct computation,
\begin{align*}
r^4F z^2 |\nabla r|^2&=\left( r^2 F+r^{-2}F^{-1}\left( \frac{\partial s}{\partial y} \right)^2 \right)\left( \frac{\partial s}{\partial x} \right)^2 +\left( r^2+r^{-2}F^{-1}\left( \frac{\partial s}{\partial x} \right)^2 \right)\left( \frac{\partial s}{\partial y} \right)^2\\
&+2\left(-r^{-2}F^{-1}\frac{\partial s}{\partial x}\frac{\partial s}{\partial y}\right)\frac{\partial s}{\partial x}\frac{\partial s}{\partial y}\\
&=r^2 F\left(\frac{\partial s}{\partial x} \right)^2+r^2  \left(\frac{\partial s}{\partial y} \right)^2=r^6 F^{2}(z^2-1).
\end{align*}
Here we use \eqref{def:z} for the last equality. Then \eqref{equ:gradr} follows by rearranging terms. \\

We turn to \eqref{equ:grady}. Through a direct computation,
\begin{align*}
\vert\nabla y\vert^2=g^{yy} &= r^{-4}F^{-1}z^{-2}\left( r^2+r^{-2}F^{-1}\left( \frac{\partial s}{\partial x} \right)^2 \right)\\
&= r^{-4}F^{-1}z^{-2}  \cdot r^2 \left(z^2-r^{-4}F^{-2}\left(\frac{\partial s}{\partial y} \right)^2  \right) \\
&= r^{-2}F^{-1}-r^{-6}F^{-3}z^{-2}\left(\frac{\partial s}{\partial y} \right)^2 .
\end{align*}
\end{proof}

Next, we calculate the second fundamental form of $\Sigma$.
\begin{lemma}
\label{lemma:secondform}
The second fundamental form of $\Sigma$ is given by
\begin{equation}\label{equ:h}
\begin{split}
   N\cdot h_{xx} &=   -r^{-2}F^{-1}\frac{\partial^2 s}{\partial x^2} +(3r^{-3}F^{-1}+2^{-1}r^{-2}F^{-2}F') \left(\frac{\partial s}{\partial x}\right)^2 +r , \\
  N\cdot  h_{yy} &=    - r^{-2}F^{-1}\frac{\partial^2s}{\partial y^2} + \frac 32 r^{-4}F^{-2}(r^2F)' \left(\frac{\partial s}{\partial y}\right)^2  + \frac 12 (r^2F)' , \\
   N\cdot h_{xy} &=  - r^{-2}F^{-1}\frac{\partial^2s}{\partial x\partial y} +(r^{-3}F^{-1}+r^{-4}F^{-2}(r^2F)')\frac{\partial s}{\partial x}\frac{\partial s}{\partial y}.
\end{split}
\end{equation}
\end{lemma}
\begin{proof}
Recall that $\bar{\nabla}$ is the Levi-Civita connection of the Ads-Melvin metric $\bar{g}$. From \eqref{equ:tangent}, we compute
\begin{align*}
    \bar{\nabla}_{\frac{\partial\varphi}{\partial x} }\frac{\partial\varphi}{\partial x} =&\bar{\nabla}_{\frac{\partial\varphi}{\partial x} }\left( \frac{\partial s}{\partial x}\frac{\partial}{\partial r}+\frac{\partial}{\partial x} \right)\\
    =&\frac{\partial^2 s}{\partial x^2}\frac{\partial}{\partial r}+\frac{\partial s}{\partial x}    \bar{\nabla}_{\frac{\partial\varphi}{\partial x} }\frac{\partial}{\partial r}+    \bar{\nabla}_{\frac{\partial\varphi}{\partial x} }\frac{\partial}{\partial x}\\
    =&\frac{\partial^2 s}{\partial x^2}\frac{\partial}{\partial r}+\left(\frac{\partial s}{\partial x} \right)^2   \bar{\nabla}_{\frac{\partial}{\partial r}  }\frac{\partial}{\partial r}+ \frac{\partial s}{\partial x}   \bar{\nabla}_{\frac{\partial}{\partial x}  }\frac{\partial}{\partial r}+    \frac{\partial s}{\partial x}   \bar{\nabla}_{\frac{\partial}{\partial r}  }\frac{\partial}{\partial x}+\bar{\nabla}_{\frac{\partial}{\partial x}  }\frac{\partial}{\partial x}
\end{align*}    
From \eqref{equ:Gamma}, we have
\begin{align*}
\bar{\nabla}_{\frac{\partial\varphi}{\partial x} }\frac{\partial\varphi}{\partial x}=\left(\frac{\partial^2 s}{\partial x^2}+2^{-1}r^2 F (r^{-2}F^{-1})'\left(\frac{\partial s}{\partial x}\right)^2    - r^3F \right)\frac{\partial}{\partial r} + 2r^{-1}\frac{\partial s}{\partial x}  \frac{\partial}{\partial x}.
\end{align*}
Similar calculation yields
\begin{align*}
    \bar{\nabla}_{\frac{\partial\varphi}{\partial y}}\frac{\partial\varphi}{\partial y} = \left(\frac{\partial^2s}{\partial y^2}+2^{-1}r^2F(r^{-2}F^{-1})' \left(\frac{\partial s}{\partial y}\right)^2 - 2^{-1} r^2F(r^2F)' \right)\frac{\partial}{\partial r}  
     + r^{-2}F^{-1}(r^2F)'\frac{\partial s}{\partial
    y} \frac{\partial}{\partial y}
\end{align*}
and
\begin{align*}
\bar{\nabla}_{\frac{\partial\varphi}{\partial x} }\frac{\partial\varphi}{\partial y} =  \left(\frac{\partial^2 s}{\partial x\partial y}+ 2^{-1}r^2F(r^{-2}F^{-1})' \frac{\partial s}{\partial x}\frac{\partial s}{\partial y}\right) \frac{\partial}{\partial r} + r^{-1}\frac{\partial s}{\partial y} \frac{\partial}{\partial x} 
     +2^{-1}r^{-2}F^{-1}(r^2F)' \frac{\partial s}{\partial
    x} \frac{\partial}{\partial y}.
\end{align*}
The second fundamental form can now be computed through taking the inner product with the unit normal vector $-\nu$ given in \eqref{equ:normal}. 
\end{proof}
\begin{lemma}
\label{lemma:H}
The mean curvature $H$ of $\Sigma$ is given by
\begin{equation}\label{equ:H}
\begin{split}
        {r^6F^2N^3}H=&\bigg(- 1- r^{-4}F^{-2}\left(\frac{\partial s}{\partial y}\right)^2 \bigg)\frac{\partial^2 s}{\partial x^2}+\bigg(-F^{-1}
     -r^{-4}F^{-2}\left(\frac{\partial s}{\partial x}\right)^2 \bigg)\frac{\partial^2 s}{\partial y^2}\\
     &+2r^{-4}F^{-2} \frac{\partial s}{\partial x}\frac{\partial s}{\partial y} \frac{\partial^2 s}{\partial x\partial y}+ (4r^{-1}+F^{-1}F') \left(\frac{\partial s}{\partial x}\right)^2\\
        & +
    \left(4r^{-1}F^{-1}+\frac{3}{2}F^{-2}F'\right)\left(\frac{\partial s}{\partial y}\right)^2+2r^3F +\frac 12 r^4F'.
\end{split}
\end{equation}

\end{lemma}
\begin{proof}
The assertion follows by plugging \eqref{equ:inversemetric} and \eqref{equ:h} into $$H = g^{xx}h_{xx}+g^{xy}h_{xy}+g^{yx}h_{yx}+g^{yy}h_{yy}. $$
\end{proof}

Lastly, we calculate the relation between the mean curvature $H$ and $\Delta r$. To do so, we record the Hessian of $r$ below.

\begin{lemma}
We have the equalities
\begin{equation}\label{equ:hessbar_r}
\bar{\nabla}_a\bar{\nabla}_b r=r F \bar{g}_{ab}+2^{-1}F^{-1}F' \bar{\nabla}_a r\bar{\nabla}_b r+2^{-1}r^4 FF'\bar{\nabla}_a y\bar{\nabla}_b y
\end{equation}
and
\begin{equation}\label{equ:Lapbar_r}
\bar{\Delta}r=3rF+r^2 F'.
\end{equation}
\end{lemma}
\begin{proof}
From \eqref{equ:Gamma},
\begin{align*}
    \bar{\Gamma}_{rr}^r &= \frac 12 r^2F  \left(r^{-2}F^{-1} \right)'=-r^{-1}-2^{-1}F^{-1}F', \\
    \bar{\Gamma}_{xx}^r &= -r^3F, \\
    \bar{\Gamma}_{yy}^r &= -\frac 12 r^2F \left(r^{ 2}F  \right)'=-r^{3}F^2-2^{-1}r^4FF'.
\end{align*}
Therefore,
\begin{align*}
\bar{\nabla}_a\bar{\nabla}_b r  dx^a dx^b=&\left( r^{-1}+2^{-1}F^{-1}F'\right)   dr^2+\left( r^3 F \right)   dx^2+\left( r^{3}F^2+2^{-1}r^4FF'\right)   dy^2\\
=&\big(    r^{-1}  dr^2+  r^3 F  dx^2+  r^{3}F^2     dy^2 \big)+ 2^{-1}F^{-1}F'  dr^2 +2^{-1}r^4FF' dy^2\\
=&rF\bar{g}_{ab} dx^a dx^b+ 2^{-1}F^{-1}F'  dr^2   +2^{-1}r^4FF' dy^2.
\end{align*}
Then \eqref{equ:hessbar_r} follows. \eqref{equ:Lapbar_r} can then be derived directly from \eqref{equ:hessbar_r}.
\end{proof}
\begin{lemma}
We have the equality
\begin{equation}\label{equ:Lap_r}
\begin{split}
\Delta r=-&rF^{1/2}z^{-1}H+  2rF+2^{-1}r^2F' \\
&+ 2^{-1}r^2  F'(1-z^{-2})-2^{-1}r^{-2}F^{-2}F' z^{-2}\left( \frac{\partial s}{\partial y} \right)^2  .
\end{split}
\end{equation}
\end{lemma}
\begin{proof}
From $\bar{g}(\bar{\nabla}r,\nu)=rF^{1/2}z^{-1}$ and \eqref{equ:hessbar_r},  we derive
\begin{align*}
\nabla_i\nabla_j r=&\bar{\nabla}_i\bar{\nabla}_j r-\bar{g}(\bar{\nabla}r,\nu) h_{ij}\\
                  = &r F  {g}_{ij}+2^{-1}F^{-1}F'  {\nabla}_i r {\nabla}_j r+2^{-1}r^4 FF' {\nabla}_i y {\nabla}_j y-rF^{1/2}z^{-1} h_{ij}.
\end{align*}
Taking the trace, we obtain
\begin{align*}
\Delta r=2rF+2^{-1}F^{-1}F'|\nabla r|^2 +2^{-1}r^4 FF' |\nabla y|^2 -rF^{1/2}z^{-1} H.	 
\end{align*} 
Using \eqref{equ:gradr} and \eqref{equ:grady}, the above becomes
\begin{align*}
\Delta r=&2rF+2^{-1}F^{-1}F'\left( r^2 F(1-z^{-2}) \right)  \\
&+2^{-1}r^4 FF' \left(r^{-2}F^{-1}-r^{-6}F^{-3}z^{-2}\left(\frac{\partial s}{\partial y} \right)^2 \right)  -rF^{1/2}z^{-1} H.
\end{align*}
Then \eqref{equ:Lap_r} follows by rearranging terms.
\end{proof}

\section{Proof of Theorem~\ref{thm:symmetry} and Theorem~\ref{thm:perturbation}}
In this section, we provide the proofs of Theorem~\ref{thm:symmetry} and Theorem~\ref{thm:perturbation}. Theorem~\ref{thm:symmetry} is separated into Proposition~\ref{prop:ysym} (the $y$-symmetry case) and Proposition~\ref{prop:xsym} ($x$-symmetry case). These two propositions are proved in Section~\ref{sec:3.1} and Section~\ref{sec:3.2} respectively. In Section~\ref{sec:3.3}, we prove Theorem~\ref{thm:perturbation}.\\

Recall that $\Sigma$ is defined as the surface $r=s(x,y)$ for a smooth function $s$ defined on $T^2$.
\begin{lemma}
\label{lemma:specialcase}
We have the equality
\begin{align}\label{main-parametrized}
    Q\left(\Sigma\right) - P_xP_y\left(2r_s^3 - 2^{-1}\right)= \int_0^{P_x}\int_0^{P_y}\left(Hs^4FN - 2s^3 + 2^{-1}\right)\,   dy  dx.
\end{align}
\end{lemma}
\begin{proof}
By \eqref{volforms}, we have
\[\int_{\Sigma} H r\, dA = \int_0^{P_x}\int_0^{P_y}Hs^4FN\,  dx dy\]
and
\begin{align*} %are the s, s(x,y), and r here formatted fine
    \int_{\Omega} r\, dV &= \int_0^{P_x}\int_0^{P_y}\int_{r_{s}}^{s(x,y)}r^2\,  dr dy dx \\
    &= \int_0^{P_x}\int_0^{P_y}\frac 13\left(s(x,y)^3 - r_{s}^3\right)\, dy dx \\
    &=\frac 13\left(-P_xP_yr_s^3 +x      \int_0^{P_x}\int_0^{P_y}s^3\, dy dx\right).
\end{align*}
Therefore, $Q\left(\Sigma\right)$ is equal to
\[\int_{\Sigma} H r\, dV-6\int_{\Omega} r\, dV = 2P_xP_yr_s^3 + \int_0^{P_x}\int_0^{P_y}(Hs^4FN - 2s^3)\,  dy dx.\]
Subtracting $2(P_xP_yr_s^3-2^{-1})$ from both sides, we obtain \eqref{main-parametrized}.
\end{proof}

\subsection{Surfaces with $y$-symmetry}\label{sec:3.1}\hfill\\

We first prove Theorem \ref{thm:symmetry} assuming that $\Sigma$ has $y$-symmetry. That is, the function $s(x,y)$ is independent of $y$. We aim to show the following result.
\begin{prop}
\label{prop:ysym}
Let $\Sigma$ be the graph defined by $r=s(x)$ for a smooth periodic function $s$ with period $P_x$. Then
\[Q\left(\Sigma\right)\geq P_xP_y\left(2r_s^3 - 2^{-1}\right).\]
Equality holds if and only if $s$ is a constant function.
\end{prop}
Let us define a function that will assist in the proof of Lemma \ref{prop:ysym}.
\begin{lemma} 
\label{lemma:defs}
Let $\lambda=\lambda(r)$ be a solution to the ordinary differential equation
\begin{equation}
    \frac{d\lambda}{dr} = r^{-2}\left(F(r)\right)^{-1/2}.
\end{equation}
We abuse notation and write $\lambda(x) = \lambda(s(x))$, so that
\begin{equation}
\label{eqn:defs}
\frac{d \lambda}{d x} = r^{-2}F^{-1/2}\frac{d s}{d x}.
\end{equation}
Then
\[-r^{2}F^{1/2}\, \frac{d^2 \lambda}{d x^2} = -\frac{d^2 s}{d x^2} + 2^{-1}r^{-1}F^{-1}(4-r^{-3})\left(\frac{d s}{d x}\right)^2.\]
\end{lemma}
\begin{proof}
We compute
\begin{align*}
    \frac{d^2 \lambda}{d x^2} &= r^{-2}F^{-1/2}\frac{d^2 s}{d x^2} + \left(\frac{d s}{d x}\right)^2\frac{ d}{dr}\left(r^{-2}F^{-1/2}\right) \\
    &= r^{-2}F^{-1/2}\frac{d^2 s}{d x^2} - 2^{-1}r^{-3}F^{-3/2}\left(4F+rF'\right)\left(\frac{d s}{d x}\right)^2\\
    &= r^{-2}F^{-1/2}\frac{d^2 s}{d x^2} - 2^{-1}r^{-3}F^{-3/2}\left(4-r^{-3}\right)\left(\frac{d s}{d x}\right)^2.
\end{align*}
Multiplying both sides by $r^2F^{1/2}$, the conclusion follows.
\end{proof}
\begin{lemma}
\label{lemma:ysimp}
Define a function $\lambda=\lambda(x)$ by \eqref{eqn:defs}. Then
\begin{equation}
    Hs^4FN - 2s^3 + \frac 12 = -r^2F^{1/2}\, \frac{d^2 \lambda}{d x^2} \left(1 + \left(\frac{d \lambda}{d x}\right)^2\right)^{-1}.
\end{equation}
\end{lemma}
\begin{proof}
Since $\frac{\partial s}{\partial y}=0$, Lemma \ref{lemma:H} implies that
\[H = r^{-6}F^{-2}N^{-3}\left(2r^3F+\frac 12 r^4F' - \frac{d^2 s}{d x^2} + r^{-1}F^{-1}\left(4F+rF'\right)\left(\frac{d s}{d x}\right)^2\right).\]
Note that
\[4F + rF' = 4\left(1-r^{-3} - br^{-4}\right) + r\left(3r^{-4}+4br^{-5}\right) = 4-r^{-3}\]
and
\[s^2FN = 1 + r^{-4}F^{-1}\left(\frac{\partial s}{\partial x}\right)^2+r^{-4}F^{-2}\left(\frac{\partial s}{\partial y}\right)^2 = 1 + r^{-4}F^{-1}\left(\frac{d s}{d x}\right)^2.\]
Therefore
\begin{align*}
    Hs^4N - 2s^3 + \frac 12
    &=r^{-2}F^{-1}N^{-2}\left(2r^3F+\frac 12 r^4F' - \frac{d^2 s}{d x^2} + r^{-1}F^{-1}\left(4F+rF'\right)\left(\frac{d s}{d x}\right)^2\right) - 2r^3 + \frac 12 \\
    &= r^{-2}F^{-1}N^{-2}\left(2r^3 - \frac 12 - \frac{d^2 s}{d x^2} + r^{-1}F^{-1}(4-r^{-3})\left(\frac{d s}{d x}\right)^2\right) - 2r^3 + \frac 12 \\
    &= r^{-2}F^{-1}N^{-2}\left(-r^{-4}F^{-1}\left(2r^3 - \frac 12\right)\left(\frac{d s}{d x}\right)^2 - \frac{d^2 s}{d x^2} + r^{-1}F^{-1}(4-r^{-3})\left(\frac{d s}{d x}\right)^2\right) \\
    &= r^{-2}F^{-1}N^{-2}\left(-\frac{d^2 s}{d x^2} + 2^{-1}r^{-4}F^{-1}(4r^3 - 1)\left(\frac{d s}{d x}\right)^2\right) \\
    &= \left(-\frac{d^2 s}{d x^2} + 2^{-1}r^{-4}F^{-1}(4r^3 - 1)\left(\frac{d s}{d x}\right)^2\right)\left(1 + r^{-4}F^{-1}\left(\frac{d s}{d x}\right)^2\right)^{-1}.
\end{align*}
By Lemma \ref{lemma:defs}, the last expression is equal to
\[-r^2F^{1/2}\, \frac{d^2 \lambda}{d x^2}\left(1 + \left(\frac{d \lambda}{d x}\right)^2\right)^{-1},\]
concluding the proof.
\end{proof}
\begin{proof}[Proof of Proposition~\ref{prop:ysym}]
In view of Lemma \ref{lemma:specialcase}, it suffices to show that 
\[\int_0^{P_x}\int_0^{P_y}\left(Hs^4FN - 2s^3 + 2^{-1}\right)\,   dy  dx \geq  0.\]
Since $s$ is independent of $y$, this is equivalent to
\[\int_0^{P_x}\left(Hs^4FN - 2s^3 + \frac 12\right)\,   dx \geq 0.\]
By Lemma \ref{lemma:ysimp},
\[\int_0^{P_x}\left(Hs^4FN - 2s^3 + \frac 12\right)\,   dx = \int_0^{P_x}-\left(s^2\sqrt F\, \frac{d^2 \lambda}{d x^2}\right) \left(1 + \left(\frac{d \lambda}{d x}\right)^2\right)^{-1}\, dx, \]
which may be rewritten as
\[\int_0^{P_x}-\left(s^2\sqrt F \frac{d}{d x}\left(\arctan \frac{d \lambda}{d x}\right)\right) \, dx.\]
Through integration by parts, this is equal to
\begin{align*}
&\left(-s^2\sqrt F \arctan \frac{d \lambda}{d x}\right)\bigg\vert_{x=0}^{P_x} +\int_0^{P_x}\arctan\left(\frac{d \lambda}{d x}\right)\frac{d}{d x}(s^2\sqrt F)\, dx \\
&= \int_0^{P_x}\frac{d s}{d x}\frac{d}{d s}(s^2\sqrt F)\arctan\frac{d \lambda}{d x}\, dx \\
&= \int_0^{P_x}\frac{d \lambda}{d x}\, s^2\sqrt{F}\frac{d}{d s}(s^2\sqrt F)\arctan\frac{d \lambda}{d x}\, dx \\
&\geq 0,
\end{align*}
where the final inequality holds because the integrand is nonnegative:
\[\arctan\left(\frac{d \lambda}{d x}\right)\frac{d \lambda}{d x}, \  s^2\sqrt{F}, \ \frac{d}{d s}(s^2\sqrt F) \geq 0.\]
In order for equality to hold, we require $\frac{d\lambda}{dx}=0$ identically, so that $\lambda(s(x))$ is a constant function. Since $\lambda$ is strictly increasing, $s(x)$ must be a constant function as well. This completes the proof of our first special case.
\end{proof}

\subsection{Surfaces with $x$-symmetry}\label{sec:3.2}\hfill\\

Here, we prove Theorem \ref{thm:symmetry} assuming that $\Sigma$ has $x$-symmetry.
\begin{prop}
\label{prop:xsym}
Let $\Sigma$ be the graph defined by $r=s(y)$ for a periodic smooth function $s$ with period $P_y$. Then
\[Q\left(\Sigma\right)\geq P_xP_y\left(2r_s^3 - 2^{-1}\right).\]
Equality holds if and only if $s$ is a constant function.
\end{prop}
We follow a similar strategy as the proof for $y$-symmetric surfaces.
\begin{lemma}
\label{lemma:defs2}
Let $\lambda = \lambda(r)$ be a solution to the ordinary differential equation
\begin{equation}
    \frac{d\lambda}{dr} = r^{-2}\left(F(r)\right)^{-1}.
\end{equation}
We abuse notation and write $\lambda(y) = \lambda(s(y))$, so that
\begin{equation}
\label{equ:defs2}
\frac{d \lambda}{d y} = r^{-2}F^{-1}\frac{d s}{d y}.    
\end{equation}
Then
\[-r^2\,\frac{d^2 \lambda}{d y^2} = -F^{-1}\frac{d^2 s}{d y^2} + r^{-1}F^{-2}(2F+rF')\left(\frac{d s}{d y}\right)^2.\]
\end{lemma}
\begin{proof}
We compute
\begin{align*}
    \frac{d^2 \lambda}{d y^2} &= r^{-2}F^{-1}\frac{d^2 s}{d y^2} + \left(\frac{d s}{d y}\right)^2\frac{d}{ dr}\left(\frac{1}{r^2F}\right) \\
    &= r^{-2}F^{-1}\frac{d^2 s}{d y^2} + \left(\frac{d s}{d y}\right)^2\frac{d}{dr}\left(\frac{1}{r^2F}\right) \\
    &= r^{-2}F^{-1}\frac{d^2 s}{d y^2} - r^{-3}F^{-2}(2F+rF')\left(\frac{d s}{d y}\right)^2. 
\end{align*}
Multiplying both sides by $-r^2$ implies the desired result.
\end{proof}
\begin{lemma}
\label{lemma:xsimp}
Define a function $\lambda=\lambda(y)$ by \eqref{equ:defs2}. Then
\begin{equation}
   Hs^4FN - 2s^3 + \frac 12 = -r^2\,\frac{d^2 \lambda}{d y^2} \left(1 + \left(\frac{d \lambda}{d y}\right)^2\right)^{-1}.
\end{equation}
\end{lemma}
\begin{proof}
Since $\frac{\partial s}{\partial x}=0$, Lemma \ref{lemma:H} implies that
\[H = r^{-6}F^{-2}N^{-3}\left(2r^3F + \frac 12 r^4 F' - F^{-1} \frac{d^2 s}{d y^2}+\frac{8F+3rF'}{2rF^2}\left(\frac{d s}{d y}\right)^2\right).\]
Thus
\begin{align*}
Hs^4FN - 2s^3 + \frac 12 &=r^{-2}F^{-1}N^{-2}\left(2r^3 - \frac 12 - F^{-1} \frac{d^2 s}{d y^2}+\frac{8F+3rF'}{2rF^2}\left(\frac{d s}{d y}\right)^2\right) - 2r^3 + \frac 12 \\
&= r^{-2}F^{-1}N^{-2}\left(-r^{-4}F^{-2}\left(2r^3 - \frac 12\right)\left(\frac{d s}{d y}\right)^2 - F^{-1} \frac{d^2 s}{d y^2}+\frac{8F+3rF'}{2rF^2}\left(\frac{d s}{d y}\right)^2\right) \\
&= r^{-2}F^{-1}N^{-2}\left(-F^{-1} \frac{d^2 s}{d y^2}+\frac{2F+rF'}{rF^2}\left(\frac{d s}{d y}\right)^2\right) \\
&= \left(-F^{-1} \frac{d^2 s}{d y^2}+\frac{2F+rF'}{rF^2}\left(\frac{d s}{d y}\right)^2\right)\left(1 + r^{-4}F^{-2}\left(\frac{d s}{d y}\right)^2\right)^{-1}.
\end{align*}
By Lemma \ref{lemma:defs2}, the final expression is equal to
\[-r^2\,\frac{d^2 \lambda}{d y^2} \left(1 + \left(\frac{d \lambda}{d y}\right)^2\right)^{-1},\]
as desired.
\end{proof}
\begin{proof}[Proof of Proposition ~\ref{prop:xsym}]
By Lemma \ref{lemma:specialcase}, it suffices to show that
\[\int_0^{P_x}\int_0^{P_y}\left(Hs^4FN - 2s^3 + 2^{-1}\right)\,   dy  dx \geq 0.\]
Since $s$ is independent of $x$, this is equivalent to
\[\int_0^{P_y}\left(Hs^4FN - 2s^3 + \frac 12\right)\,   dy \geq 0.\]
By Lemma \ref{lemma:xsimp}, we have
\begin{align*}
    \int_0^{P_y}\left(Hs^4FN - 2s^3 + \frac 12\right)\,   dy &= \int_0^{P_y} -s^2\,\frac{d^2 \lambda}{d y^2} \left(1 + \left(\frac{d \lambda}{d y}\right)^2\right)^{-1}  dy \\
    &=\int_0^{P_y} -s^2\frac{d}{d y}\arctan\left(\frac{d \lambda}{d y}\right)  dy.
\end{align*}
Using integration by parts, this is equal to
\begin{align*}
    &\left(-s^2\arctan\left(\frac{d \lambda}{d y}\right)\right)\bigg\vert_{y=0}^{P_y} + \int_0^{P_y}2s\frac{d s}{d y}\arctan\left(\frac{d \lambda}{d y}\right) dy \\
    &= \int_0^{P_y}2s^3F\frac{d \lambda}{d y}\arctan\left(\frac{d \lambda}{d y}\right) dy \\
    &\geq 0,
\end{align*}
where the final inequality holds since the integrand is nonnegative:
\[2s^3F, \ \frac{d\lambda}{d y}\arctan \left(\frac{d\lambda}{d y}\right) \geq 0.\]
In order for equality to hold, we require $\frac{d\lambda}{dy}=0$ identically, so that $\lambda(s(y))$ is a constant function. Since $\lambda$ is strictly increasing, $s(y)$ must be a constant function as well. 
\end{proof}

\subsection{Perturbation}\label{sec:3.3}\hfill \\

We now prove Theorem \ref{thm:perturbation}, where the surface $\Sigma$ is a small perturbation of a constant-height surface. For reference, we reproduce the statement below.
\begin{prop}
\label{prop:perturbation}
Fix $r_0>r_s$ and a smooth function $\phi(x,y)$ on the torus. For $\varepsilon$ small, consider 
\begin{equation}
\label{perturbation_def}
\tilde{s}_\varepsilon(x,y) = r_0 + \varepsilon \phi(x,y), 
\end{equation} 
and the graph $\tilde{\Sigma}_\epsilon$ given by
\begin{equation}
\tilde{\Sigma}_\varepsilon=\{ (\tilde{s}_\varepsilon(x,y),x,y)\, |\, (x,y)\in T^2 \}.
\end{equation}
Then we have
\begin{equation}
\label{firstorder}
\frac{d}{d\varepsilon}Q(\tilde{\Sigma}_{\varepsilon})\bigg|_{\varepsilon=0} =0,
\end{equation}
and
\begin{equation}
\label{secondorder}
\frac{d^2}{d\varepsilon^2}Q(\tilde{\Sigma}_{\varepsilon})\bigg|_{\varepsilon=0} \geq 0.
\end{equation}
Moreover, equality holds in \eqref{secondorder} if and only if $\phi$ is a constant function.
\end{prop}
For later convenience, we write
\begin{align}
    F_0=F(r_0),\ F'_0=F'(r_0)
\end{align}
for the values of $F$ and $F'$ on $\tilde{\Sigma}_0$. 
Let $d\tilde{A}_{\varepsilon}$ denote the area form of $\tilde{\Sigma}_{\varepsilon}$ and $\tilde{\Omega}_{\varepsilon}$ denote the region enclosed by $\tilde{\Sigma}_{\varepsilon}$. We use the tilde in $\tilde{\Sigma}_{\varepsilon}$ and analogous expressions to distinguish them from $\Sigma_t$ in the next section, where we consider the normal flow. For brevity, other quantities of $\tilde{\Sigma}_{\varepsilon}$ will not carry the tildes or subscripts. For instance, we still write $H$ for the mean curvature of $\tilde{\Sigma}_{\varepsilon}$.

\begin{lemma}
\label{lemma:perturbationequations} 
On the surface $\tilde{\Sigma}_\epsilon$, the following expansion holds.
\begin{equation}
\label{r2FN2}
\frac{1}{r^2FN^2} = 1-\varepsilon^2\left(\frac{\partial \phi}{\partial x}\right)^2\frac{1}{r^4F_0}-\varepsilon^2\left(\frac{\partial \phi}{\partial y}\right)^2\frac{1}{r^4F^2_0} + O(\varepsilon^3),
\end{equation}
\begin{equation}
\label{F}
\frac 1F = \frac{1}{F_0} - \varepsilon \frac{\phi F'_0}{F_0^2}+O(\varepsilon^2).
\end{equation}
\end{lemma}
\begin{proof}
Since all partial derivatives of $r$ with respect to $\varepsilon$ are first-order in $\varepsilon$, equation \eqref{r2FN2} follows from the definition of $N$ as
\[r^2FN^2 = 1+\left(\frac{\partial s}{\partial x}\right)^2\frac{1}{r^4F}+\left(\frac{\partial s}{\partial y}\right)^2\frac{1}{r^4F^2}.\]
Similarly, 
\begin{align*}
    F &= F(r_0 + \varepsilon\phi) = F_0 + \varepsilon\phi F'_0 + O(\varepsilon^2). \\
\end{align*}
Taking the reciprocal, we have 
\begin{align*}
\frac 1F &= \frac{1}{F_0} - \varepsilon \frac{\phi F'_0}{F_0^2}+O(\varepsilon^2).
\end{align*}
\end{proof}
\begin{lemma}
\label{lemma:perturbationsimp}
On the surface $\tilde{\Sigma}_\varepsilon$, we have the following expansion in $\varepsilon$:
\begin{align*}
Hr^4FN &- 2r^3 + \frac 12  \\ &= -\varepsilon\left(\frac{\partial^2\phi}{\partial x^2}+\frac{1}{F_0}\frac{\partial^2\phi}{\partial y^2}\right) + \varepsilon^2 \left(\frac{4F_0+r_0F'_0}{2r_0F_0}\left(\frac{\partial \phi}{\partial x}\right)^2+\frac{2F_0+r_0F'_0}{r_0F_0^2}\left(\frac{\partial \phi}{\partial y}\right)^2+\frac{\phi F'_0}{F_0^2}\frac{\partial^2\phi}{\partial y^2}\right) + O(\varepsilon^3).
\end{align*}
\end{lemma}
\begin{proof}
Note that
\[\frac{4F+rF'}{2rF}=\frac{4F_0+r_0F'_0}{2r_0F_0}+O(\varepsilon), \ \frac{2F+rF'}{rF^2}=\frac{2F_0+r_0F'_0}{r_0F_0^2}+O(\varepsilon).\]
Using these equations and Lemma \ref{lemma:perturbationequations}, we compute
\begin{align*}
    Hr^4FN = & \ \frac{1}{r^2FN^2}\left(2r^3F + \frac 12 r^4F' - \frac{\partial^2 s}{\partial x^2} - \frac 1F \frac{\partial^2 s}{\partial y^2}+\frac{4F+rF'}{rF}\left(\frac{\partial s}{\partial x}\right)^2+\frac{8F+3rF'}{2rF^2}\left(\frac{\partial s}{\partial y}\right)^2+O(\varepsilon^3)\right) \\
    = & \ 2r^3F + \frac 12 r^4F' - \frac{\partial^2 s}{\partial x^2} - \frac 1F \frac{\partial^2 s}{\partial y^2}+\frac{4F+rF'}{rF}\left(\frac{\partial s}{\partial x}\right)^2+\frac{8F+3rF'}{2rF^2}\left(\frac{\partial s}{\partial y}\right)^2 \\
    & - \left(2r^3F + \frac 12 r^4F'\right)\left(\left(\frac{\partial s}{\partial x}\right)^2\frac{1}{r^4F}+\left(\frac{\partial s}{\partial y}\right)^2\frac{1}{r^4F^2}\right) + O(\varepsilon^3) \\
    = & \ 2r^3F + \frac 12 r^4F' - \frac{\partial^2 s}{\partial x^2} - \frac 1F \frac{\partial^2 s}{\partial y^2}+\frac{4F+rF'}{2rF}\left(\frac{\partial s}{\partial x}\right)^2+\frac{2F+rF'}{rF^2}\left(\frac{\partial s}{\partial y}\right)^2+O(\varepsilon^3) \\
    = & \ 2r^3 - \frac 12 - \frac{\partial^2 s}{\partial x^2} - \frac 1F \frac{\partial^2 s}{\partial y^2}+\frac{4F+rF'}{2rF}\left(\frac{\partial s}{\partial x}\right)^2+\frac{2F+rF'}{rF^2}\left(\frac{\partial s}{\partial y}\right)^2+O(\varepsilon^3).
\end{align*}
%fix reference
Combining this with \eqref{perturbation_def}, it follows that
\begin{align*}
    %can split with words
    & Hr^4FN - 2r^3 + \frac 12 \\
    = &- \frac{\partial^2 s}{\partial x^2} - \frac 1F \frac{\partial^2 s}{\partial y^2}+\frac{4F+rF'}{2rF}\left(\frac{\partial s}{\partial x}\right)^2+\frac{2F+rF'}{rF^2}\left(\frac{\partial s}{\partial y}\right)^2+O(\varepsilon^3) \\
    %= &- \frac{\partial^2 r}{\partial x^2} - \frac 1F \frac{\partial^2 r}{\partial y^2}+\frac{4F+rF'}{2rF}\left(\frac{\partial r}{\partial x}\right)^2+\frac{2F+rF'}{rF^2}\left(\frac{\partial r}{\partial y}\right)^2 + O(\varepsilon)^3 \\
    = &- \varepsilon\frac{\partial^2\phi}{\partial x^2} - \varepsilon\frac{\partial^2\phi}{\partial y^2}\left(\frac{1}{F_0} - \varepsilon \frac{\phi F'_0}{F_0^2}+O(\varepsilon^2)\right) + \varepsilon^2\frac{4F_0+rF'_0}{2rF_0}\left(\frac{\partial \phi}{\partial x}\right)^2+\varepsilon^2\frac{2F_0+rF'_0}{rF_0^2}\left(\frac{\partial \phi}{\partial y}\right)^2 + O(\varepsilon^3) \\ 
    = &- \varepsilon\left(\frac{\partial^2\phi}{\partial x^2}+\frac{1}{F_0}\frac{\partial^2\phi}{\partial y^2}\right) + \varepsilon^2 \left(\frac{4F_0+rF'_0}{2rF_0}\left(\frac{\partial \phi}{\partial x}\right)^2+\frac{2F_0+rF'_0}{rF_0^2}\left(\frac{\partial \phi}{\partial y}\right)^2+\frac{\phi F'_0}{F_0^2}\frac{\partial^2\phi}{\partial y^2}\right) + O(\varepsilon^3)\\
    = &- \varepsilon\left(\frac{\partial^2\phi}{\partial x^2}+\frac{1}{F_0}\frac{\partial^2\phi}{\partial y^2}\right) + \varepsilon^2 \left(\frac{4F_0+r_0F'_0}{2r_0F_0}\left(\frac{\partial \phi}{\partial x}\right)^2+\frac{2F_0+r_0F'_0}{r_0F_0^2}\left(\frac{\partial \phi}{\partial y}\right)^2+\frac{\phi F'_0}{F_0^2}\frac{\partial^2\phi}{\partial y^2}\right) + O(\varepsilon^3),
\end{align*}
as desired.
\end{proof}
\begin{proof}[Proof of Proposition~\ref{prop:perturbation}]
By Lemma \ref{lemma:perturbationsimp}, we have
\begin{align*}
    & \frac{d}{d\varepsilon}\left(\int_{\tilde{\Sigma}_\varepsilon} H r\, dV_{\tilde{\Sigma}_\varepsilon}-6\int_{\tilde{\Omega}_\varepsilon} r\, dV_{\tilde{\Omega}_\varepsilon}\right) \\ 
    = & -\int\limits_0^{P_x}\int\limits_0^{P_y}\left(\frac{\partial^2\phi}{\partial x^2}+\frac{1}{F_0}\frac{\partial^2\phi}{\partial y^2}\right)\,\ dy dx \\
    = & -\int\limits_0^{P_y}\int\limits_0^{P_x}\frac{\partial^2\phi}{\partial x^2}\, dx dy -\frac{1}{F_0}\int\limits_0^{P_x}\int\limits_0^{P_y}\frac{\partial^2\phi}{\partial y^2}\, dy dx \\
    = & - \int\limits_0^{P_y}\left(\frac{\partial\phi}{\partial x}\bigg\vert_{x=0}^{P_x}\right) dy - \frac{1}{F_0}\int\limits_0^{P_x}\left(\frac{\partial\phi}{\partial y}\bigg\vert_{x=0}^{P_y}\right) dx \\
    = & \ 0.
\end{align*}
Similarly,
\begin{align*}
    & \frac{d^2}{d\varepsilon^2}\left(\int_{\tilde{\Sigma}_\varepsilon} H r\, dV_{\tilde{\Sigma}_\varepsilon}-6\int_{\tilde{\Omega}_\varepsilon} r\, dV_{\tilde{\Omega}_\varepsilon}\right) \\
    = &\ 2\int\limits_0^{P_x}\int\limits_0^{P_y}\left(\frac{4F_0+r_0F'_0}{2r_0F_0}\left(\frac{\partial \phi}{\partial x}\right)^2+\frac{2F_0+r_0F'_0}{r_0F_0^2}\left(\frac{\partial \phi}{\partial y}\right)^2+\frac{\phi F'_0}{F_0^2}\frac{\partial^2\phi}{\partial y^2}\right) dy dx.
\end{align*}
We compute that
\[\int\limits_0^{P_y}\frac{\phi F'_0}{F_0^2}\frac{\partial^2 \phi}{\partial y^2}\, dy = \frac{F'_0}{F_0^2}\left(\phi\frac{\partial\phi}{\partial y}\bigg\vert_{y=0}^{P_y} - \int\limits_0^{P_y}\left(\frac{\partial \phi}{\partial y}\right)\, dy\right) = -\frac{F'_0}{F_0^2}\int\limits_{0}^{P_y}\left(\frac{\partial\phi}{\partial y}\right)^2 \, dy.\]
Therefore, the second derivative is
\[2\int\limits_0^{P_x}\int\limits_0^{P_y}\left(\frac{4F_0+r_0F'_0}{2r_0F_0}\left(\frac{\partial \phi}{\partial x}\right)^2+\frac{2F_0}{r_0F_0^2}\left(\frac{\partial \phi}{\partial y}\right)^2\right) dy dx,\]
which is nonnegative since the integrand is nonnegative. Moreover, equality holds if and only if $\frac{\partial\phi}{\partial x}=\frac{\partial\phi}{\partial y}=0$, i.e. $\phi$ is a constant function. This completes the proof of Theorem \ref{thm:perturbation}.
\end{proof}

\section{Normal Flow}

In this section, we deform a surface $\Sigma$ by a normal flow and study the evolution of  $Q(\Sigma)$ defined in \eqref{def:Q}. Let $\rho$ be a function defined on $M$. Let $\varphi:[0,T_0)\times T^2\to M$ be a family of embedded graphs satisfying
\begin{align}\label{equ:flow}
\frac{\partial  \varphi}{\partial t } = \rho\nu.
\end{align}
Here $\nu$ is the outward unit normal. Later we will specialize to the speed $\rho=r^{-1}$ and consider
\begin{align}\label{equ:trueflow}
\frac{\partial  \varphi}{\partial t } = r^{-1}\nu.
\end{align}
Nevertheless, we will keep $\rho$ unspecified in computing evolution equations in order to make its role apparent. We write $\Sigma_t$ for the image of $\varphi(t,\cdot)$ and write $dA_t$ for the area form of $\Sigma_t$. Recall that the quantity $Q(\Sigma_t)$ is defined in \eqref{def:Q}. % where $\rho = r^{-1}$. Assume the flow begins at a time $t=t_0>0$. Our goal is to prove that
%\[\int_{\Sigma} H r\, dV_{\Sigma}-6\int_{\Omega} r\, dV_{\Omega}\geq P_xP_y(2r_s^3-2^{-1}).\]
The goal of this section is to prove the following monotonicity formula.
\begin{prop}
\label{pro:monotonicity}
Under the flow \eqref{equ:trueflow}, it holds that
\begin{equation}\label{equ:monotone}
 \frac{d}{d t}Q(\Sigma_t)  \leq 0. 
\end{equation}
Moreover, equality holds if and only if $\Sigma_t$ is a coordinate torus.
\end{prop}
In Section~\ref{sec:4.1}, we record several ingredients that are needed to prove Proposition~\ref{pro:monotonicity}. These ingredients include the evolution equations, the Gauss equation and a static inequality for the AdS-Melvin space. In Section~\ref{sec:4.2} we combine the formulas in Section~\ref{sec:4.1} to finish the proof of Proposition~\ref{pro:monotonicity}.
\subsection{Basic formulas}\label{sec:4.1}\hfill \\

We start with the evolution equations along the flow \eqref{equ:flow}. The evolution equations for the area form and the mean curvature are given by
\begin{equation}\label{equ:evoarea}
\frac{\partial }{\partial t}\, dA_t=\rho H\, d A_t
\end{equation}
and
\begin{equation}\label{equ:evoH}
\frac{\partial H}{\partial t} = -\Delta \rho - \rho\,  \overline{\textup{Ric}} (\nu,\nu ) - \rho\vert A\vert^2.
\end{equation}
Here $\overline{\textup{Ric}}$ stands for the Ricci curvature of the ambient space. The derivations of the above equations are provided in Appendix~\ref{sec:B} for  the reader's convenience. \\

Next, we record the Gauss equation. Let $R$ and $\bar{R}$ be the scalar curvature of $\Sigma$ and $M$ respectively. Then we have 
\begin{equation}\label{eq:4.2}
R-H^2+\vert A\vert^2 = \bar R - 2\overline{\text{Ric}}(\nu,\nu).
\end{equation}
Recall that $\nabla$ and $\bar{\nabla}$ are the Levi-Civita connection on $\Sigma_t$ and $M$ respectively. Given a smooth function $\psi$ on $M$, we have
\begin{equation}\label{eq:4.3}
\Delta \psi = \bar\Delta\psi - \bar\nabla^2 \phi(\nu,\nu) - H\bar{g}( \bar\nabla\psi, \nu) .
\end{equation}
%The proof of \eqref{eq:4.3} is also provided in the Appendix \ref{sec:A}.\\

Lastly, we need a static inequality for $r$.
\begin{lemma}\label{lem:static}
In the AdS-Melvin space with metric $\bar{g}$ given by \eqref{metric}, the coordinate function $r$ satisfies
\begin{equation}\label{eq:4.4}
  \bar\nabla_a\bar\nabla_b r - \bar\Delta r\cdot \bar g_{ab} - r\cdot \bar R_{ab}\geq -\frac{2b}{r^3}\bar g_{ab}.
\end{equation}
\end{lemma} 

\begin{proof}
Combining \eqref{equ:hessbar_r}, \eqref{equ:Lapbar_r} and \eqref{equ:Ric}, we have
\begin{align*}
    \bar\nabla_a\bar\nabla_b r - \bar\Delta r\cdot \bar g_{ab} - r\cdot \bar R_{ab}=\left(2F'+2^{-1}r F''\right) (F^{-1}\bar{\nabla}_ar\bar{\nabla}_br +r^4 F\, \bar{\nabla}_ay\bar{\nabla}_by ).
\end{align*}
From \eqref{def:F}, $2F'+2^{-1}r F''=-2b r^{-5}$. Therefore,
\begin{align*}
 \bar\nabla_a\bar\nabla_b r - \bar\Delta r\cdot \bar g_{ab} - r\cdot \bar R_{ab}=-\frac{2b}{r^3}(r^{-2}F^{-1}\,\bar{\nabla}_a r\bar{\nabla}_b r +r^2 F\, \bar{\nabla}_a y\bar{\nabla}_b y ).
\end{align*}
From \eqref{metric}, this implies
\begin{equation*}
  \bar\nabla_a\bar\nabla_b r - \bar\Delta r\cdot \bar g_{ab} - r\cdot \bar R_{ab}\geq -\frac{2b}{r^3}\bar g_{ab}.
\end{equation*}
\end{proof}

\subsection{Monotonicity}\label{sec:4.2} \hfill \\

We prove in this section the monotonicity inequality \eqref{equ:monotone}. Combining the formulas in Section~\ref{sec:4.1}, we calculate $\frac{d}{dt}Q(\Sigma_t)$ with a general speed as the following.
\begin{lemma}
\label{lemma:monotonicitysimp}
Under the flow \eqref{equ:flow}, it holds that
\begin{equation}\label{equ:monomid}
 \frac{d}{d t}Q(\Sigma_t)  \leq \int_{\Sigma_t}  \rho(-2\Delta r -   r R)\, dA_t. 
\end{equation}
\end{lemma}
\begin{proof}
We compute
\begin{align*}
\frac{d}{d t} Q(\Sigma_t) &=    \frac{d}{d t}\left(\int_{\Sigma_t} H r\, dA_t-6\int_{\Omega_t} r\, dV\right) \\
    &= \int_{\Sigma_t}\frac{\partial H}{\partial t} r+H\frac{\partial r}{\partial t}\, dA_t + \int_{\Sigma_t}Hr\, \frac{\partial }{\partial t}dA_t - 6\int_{\Sigma_t} \rho r\, dV .
\end{align*}
From \eqref{equ:evoH}, \eqref{equ:flow} and \eqref{equ:evoarea}, the above equals
\begin{align*}
& \int_{\Sigma_t} \left((-\Delta \rho - \rho\vert A\vert^2 -\rho \overline{\text{Ric}}(\nu,\nu))r + H\bar{g}( \bar\nabla r,\nu)\rho + Hr\cdot\rho H - 6\rho r\,\right) dA_t.
\end{align*}
Applying the Gauss equation \eqref{eq:4.2} and integration by parts, this is equal to
\[\int_{\Sigma_t} \rho\left(-\Delta r +   r\overline{\text{Ric}}(\nu,\nu)+ H\bar{g}( \bar\nabla r,\nu)  - 6  r +   r R -   r\bar R\right)\, dA_t.\]
Combining \eqref{eq:4.3} and \eqref{eq:4.4}, we have
\[H\bar{g}( \bar\nabla r,\nu) + \Delta r+    r\overline{\text{Ric}}(\nu,\nu)\leq \frac{2b }{r^3}.\]
Therefore,
\begin{align*}
    &\frac{d}{d t}Q(\Sigma_t) \leq   \int_{\Sigma_t} \rho\left( -2 \Delta r -  r R -   r \bar R - 6  r + \frac{2b }{r^3}\right)dA_t	.
\end{align*}
From \eqref{equ:R}, $ -  r\bar R - 6  r + \frac{2b }{r^3} = 0 $. Then \eqref{equ:monomid} follows.
\end{proof}
We are now ready to prove Proposition~\ref{pro:monotonicity}. 
\begin{proof}[Proof of Proposition~\ref{pro:monotonicity}]
In view of \eqref{equ:monomid}, it suffices to show that
\[\int_{\Sigma_t}   \frac{2\Delta r}{r} + R\, dA_t \geq 0.\]
Observe that
\[\text{div}\left(\frac{\nabla r}{r}\right) = \frac{\Delta r}{r} - \frac{\vert \nabla r\vert ^2}{r^2}.\]
By the divergence theorem, we have
\[\int_{\Sigma_t}\left(\frac{\Delta r}{r} - \frac{\vert \nabla r\vert ^2}{r^2}\right)dA_t = \int_{\Sigma_t}\text{div}\left(\frac{\nabla r}{r}\right)dA_t = 0.\]
Notice that $\Sigma_t$ is topologically a torus. The Gauss-Bonnet theorem implies that
\[\int_{\Sigma_t} R\,dA_t = 0.\]
We then conclude that
\[\int_{\Sigma_t}  \frac{2\Delta r}{r} + R\, dA_t = \int_{\Sigma_t} \frac{2\vert\nabla r\vert^2}{r^2}\,dA_t \geq 0.\]
Moreover, the equality holds if and only if $r$ is a constant on $\Sigma_t$. 
\end{proof}

\section{Asymptotics of the Flow}
Let $\Sigma$ be an embedded surface in $M$ which is a graph over $T^2$. Let $\varphi:[0,T_0)\times T^2\to M$ be a family of embeddings satisfying \eqref{equ:trueflow} with initial data $\Sigma$. Recall that the quantity $Q(\Sigma_t)$ is defined by \eqref{def:Q}. In this section, we investigate the limit of $Q(\Sigma_t)$ assuming $T_0=\infty$. The main result is the following.
\begin{prop} 
\label{theorem:theorem5.1}
Assume $T_0=\infty$. Then 
\[\lim_{t\to\infty}Q(\Sigma_t)  = P_xP_y(2r_s^3-2^{-1}).\]
\end{prop}
Combining Proposition~\ref{pro:monotonicity} and Proposition~\ref{theorem:theorem5.1}, we can prove Theorem~\ref{theorem:main}.
\begin{proof}[Proof of Theorem~\ref{theorem:main}]
Let $\varphi:[0,T_0)\times T^2\to M$ be a family of embeddings satisfying \eqref{equ:trueflow} with initial data $\Sigma$. By assumption, we can extend $T_0$ to infinity. Proposition~\ref{pro:monotonicity} implies that $Q(\Sigma)\geq Q(\Sigma_t)$ for all $t\geq 0$. Together with Proposition~\ref{theorem:theorem5.1}, we conclude that
\begin{align*}
Q(\Sigma)\geq \lim_{t\to\infty}Q(\Sigma_t)  = P_xP_y(2r_s^3-2^{-1}). 
\end{align*}
Furthermore, if equality is achieved, then $\frac{d}{dt}Q(\Sigma_t)=0$ for all $t\in [0,\infty)$. In view of Proposition~\ref{pro:monotonicity}, this implies $\Sigma_0=\Sigma$ is a coordinate torus.
\end{proof}
The rest of the section is devoted to proving Proposition~\ref{theorem:theorem5.1}. Throughout this section, we assume that $\varphi:[0,\infty)\times T^2\to M$ is a family of embeddings which solves \eqref{equ:trueflow}. We write $\Sigma_t$ for the image of $\varphi(t,\cdot)$ and denote by $s(t,x,y)$ the height function of $\Sigma_t$. That is, under the coordinates $(r,x,y)$, 
\begin{align*}
\Sigma_t=\{(s(t,x,y),x,y)\,|\, (x,y)\in T^2\}.
\end{align*}
In Sections~\ref{sec:5.1}, \ref{sec:5.2} and \ref{sec:5.3}, we establish the $C^0$, $C^1$ and $C^2$ estimates of the height function $s(t,x,y)$. These estimates are combined in Section~\ref{sec:5.4} to prove Proposition~\ref{theorem:theorem5.1}.

\subsection{$C^0$ estimates}\label{sec:5.1}\hfill \\

We start by calculating $\frac{\partial s}{\partial t}$.

\begin{lemma}
\label{lemma:lemma5.2}
Under the flow \eqref{equ:trueflow}, we have the equality
\begin{equation}\label{equ:evo_s}
\frac{\partial s}{\partial t} = sFN.
\end{equation} 
Here $F=F(s)$ and $N$ is defined in \eqref{equ:defN}.
\end{lemma}
\begin{proof}
Notice that the normal projection of $\frac{\partial s}{\partial t}\frac{\partial}{\partial r}$ equals the velocity $r^{-1}\nu$. Therefore,
\begin{align*}
\bar{g}\left( \frac{\partial s}{\partial t}\frac{\partial}{\partial r},\nu \right)=\bar{g}\left( r^{-1}\nu,\nu \right)=r^{-1}. 
\end{align*}
From \eqref{equ:normal} and \eqref{metric}, it follows that
\[\frac{\partial s}{\partial t} = r^{-1} \bar{g}\left( \frac{\partial}{\partial r},\nu \right)^{-1} = rFN.\]
\end{proof}

Next, we show that $s$ is comparable to $t$ along the flow.
\begin{lemma} 
\label{lem:C0}
Under the flow \eqref{equ:trueflow}, there exists a constant $C_0>0$ such that for all $t\geq 0$ ,
\begin{equation}\max_{(x,y)\in T^2}\vert s(t,x,y) - t\vert \leq C_0.\end{equation}
\end{lemma}
\begin{proof}
Let $s_{\text{max}}(t)$ and $s_{\text{min}}(t)$ denote the maximum and minimum of $s(t,x,y)$ on the torus respectively. It suffices to show that $  s_{\text{max}}(t) - t $ and $  t - s_{\text{min}}(t) $ are bounded from above. We give a proof for the latter; the proof for the former is analogous.\\

At a minimum point of $s$, we have $N=s^{-1}F^{-1/2}$. Lemma \ref{lemma:lemma5.2} and \eqref{def:F} imply that
\[\frac{d s_{\min}}{d t} \geq  F^{-1/2}(s_{\min}).\]
Let $C^{-2}>0$ be the supremum of $F(r)$ for $r\in [r_s,\infty)$. The above implies that 
$$s_{\min}(t)\geq Ct+s_{\min}(0).$$ Notice that, from \eqref{def:F}, $F^{-1/2}(r)\geq 1-r^{3}$ for $r$ large enough. Therefore, there exists $t_0>0$ such that for $t\geq t_0$, we have
\begin{align*}
1-  \frac{d s_{\min}(t)}{d t}   \leq s_{\min}^{-3}(t)\leq \left( Ct+s_{\min}(0) \right)^{-3}.
\end{align*} 
Integrating the above inequality, for all $t\geq t_0$ we obtain, 
\begin{align*}
 (t - s_{\min}(t)) - (t_0 -s_{\min}(t_0))  \leq&  \int\limits_{t_0}^t \left( C\tau +s_{\min}(0) \right)^{-3}  d\tau \leq (2Cs_{\min}(0)^2)^{-1}.
\end{align*}
As a result, by setting  
\begin{align*}
C'_0=\max_{t\in [0,t_0]} (t-s_{\min}(t))+(2Cs_{\min}(0)^2)^{-1},
\end{align*}
we have $t-s_{\min}(t)\leq C_0'$ for all $t\geq 0$.

\end{proof}

\subsection{$C^1$ estimates}\label{sec:5.2} \hfill \\

Recall that $z$ is defined in \eqref{def:z} which measures the gradient of the height function. In this section, we prove that $|z^2-1|$ decays at a rate $t^{-4}$.
\begin{lemma}\label{lem:C1} 
Under the flow \eqref{equ:trueflow}, there exists a constant $C_1>0$ such that for all $t\geq 0$
\begin{equation}
\max_{(x,y)\in T^2} |z^2(t,x,y)-1|\leq C_1(t+1)^{-4}.
\end{equation}
\end{lemma}
\begin{remark}
\label{remark:C1}
From \eqref{def:z} and Lemma~\ref{lem:C0}, Lemma~\ref{lem:C1} implies that there exists a constant $C_1'$ such that for all $t\in [0,\infty)$ and $(x,y)\in T^2$, we have
\begin{equation}\label{equ:C1}
\left| \frac{\partial s}{\partial x} \right|,\left| \frac{\partial s}{\partial y} \right|\leq C_1'.
\end{equation}  
\end{remark}

We begin by deriving the evolution equation of $z^2$. 
\begin{lemma}
Under the flow \eqref{equ:trueflow}, we have 
\begin{equation}\label{equ:evo_z}
\begin{split}
\frac{\partial z^2}{\partial t}=&-4r^{-1} zF^{1/2} (z^2-1)- r^{-4}z F^{-5/2}F'  \left(\frac{\partial s}{\partial y} \right)^2   \\
&  +2r^{-4} F^{-1/2}\frac{\partial s}{\partial x} \frac{\partial z}{\partial x} +2r^{-4} F^{-3/2}\frac{\partial s}{\partial y}\frac{\partial z}{\partial y}.
\end{split}
\end{equation}
\end{lemma}
\begin{proof}
From \eqref{def:z}, we compute
\begin{align*}
\frac{\partial z^2}{\partial t} =\ &\frac{\partial r^{-4}}{\partial t} F^{-1}\left(\frac{\partial s}{\partial x} \right)^2+ r^{-4} \frac{\partial F^{-1}}{\partial t}\left(\frac{\partial s}{\partial x} \right)^2+2r^{-4} F^{-1}\frac{\partial s}{\partial x}\frac{\partial^2 s}{\partial x\partial t}\\
&+ \frac{\partial r^{-4}}{\partial t} F^{-2}\left(\frac{\partial s}{\partial y} \right)^2+ r^{-4} \frac{\partial F^{-2}}{\partial t}\left(\frac{\partial s}{\partial y} \right)^2+2r^{-4} F^{-2}\frac{\partial s}{\partial y}\frac{\partial^2 s}{\partial y\partial t}.
\end{align*}
From \eqref{def:z}, we have 
\begin{align}\label{equ:z1}
\frac{\partial r^{-4}}{\partial t} F^{-1}\left(\frac{\partial s}{\partial x} \right)^2+\frac{\partial r^{-4}}{\partial t} F^{-2}\left(\frac{\partial s}{\partial y} \right)^2=-\frac{4}{r} \frac{\partial s}{\partial t}(z^2-1)
\end{align}
and 
\begin{align}\label{equ:z2}
r^{-4} \frac{\partial F^{-1}}{\partial t}\left(\frac{\partial s}{\partial x} \right)^2+r^{-4} \frac{\partial F^{-2}}{\partial t}\left(\frac{\partial s}{\partial y} \right)^2=-F^{-1}F'\frac{\partial s}{\partial t}\left( z^2-1+r^{-4}F^{-2}\left(\frac{\partial s}{\partial y} \right)^2 \right).
\end{align}
Using \eqref{equ:z1} and \eqref{equ:z2}, we derive
\begin{align*}
\frac{\partial z^2}{\partial t}=&\left( -\frac{4}{r} \frac{\partial s}{\partial t} -F^{-1}F'\frac{\partial s}{\partial t} \right)(z^2-1) - r^{-4} F^{-3}F'\frac{\partial s}{\partial t}   \left(\frac{\partial s}{\partial y} \right)^2   \\
& +2r^{-4} F^{-1}\frac{\partial s}{\partial x}\frac{\partial^2 s}{\partial x\partial t}+2r^{-4} F^{-2}\frac{\partial s}{\partial y}\frac{\partial^2 s}{\partial y\partial t}.
\end{align*}
Next, we compute the second order derivatives of $s$. Note that we can rewrite \eqref{equ:evo_s} in terms of $z$ as
\begin{align*}
\frac{\partial s}{\partial t}= F^{1/2}z.
\end{align*}
Taking a derivative along the $x$-direction, we have
\begin{equation*}
\frac{\partial^2 s}{\partial x \partial t}=\frac{1}{2}F^{-1/2}F' z \frac{\partial s}{\partial x} + F^{1/2}\frac{\partial z}{\partial x} .
\end{equation*}
Similarly,
\begin{equation*}
\frac{\partial^2 s}{\partial y \partial t}=\frac{1}{2}F^{-1/2}F' z \frac{\partial s}{\partial y} + F^{1/2}\frac{\partial z}{\partial y} .
\end{equation*}
Therefore,
\begin{equation}\label{equ:z3}
\begin{split}
 & 2r^{-4} F^{-1}\frac{\partial s}{\partial x}\frac{\partial^2 s}{\partial x\partial t}+2r^{-4} F^{-2}\frac{\partial s}{\partial y}\frac{\partial^2 s}{\partial y\partial t}\\
 =\ &zF^{-1/2}F' (z^2-1)+2r^{-4} F^{-1/2}\frac{\partial s}{\partial x} \frac{\partial z}{\partial x} +2r^{-4} F^{-3/2}\frac{\partial s}{\partial y}\frac{\partial z}{\partial y}.
\end{split}
\end{equation}
In conclusion,
\begin{equation}
\begin{split}
\frac{\partial z^2}{\partial t}=&\left( -\frac{4}{r} \frac{\partial s}{\partial t} -F^{-1}F'\frac{\partial s}{\partial t}+zF^{-1/2}F' \right)(z^2-1)   \\
&- r^{-4} F^{-3}F'\frac{\partial s}{\partial t}   \left(\frac{\partial s}{\partial y} \right)^2  +2r^{-4} F^{-1/2}\frac{\partial s}{\partial x} \frac{\partial z}{\partial x} +2r^{-4} F^{-3/2}\frac{\partial s}{\partial y}\frac{\partial z}{\partial y}.
\end{split}
\end{equation}
Then \eqref{equ:evo_z} follows by replacing $\frac{\partial s}{\partial t}$ by $zF^{1/2}$.
\end{proof}

\begin{cor}

Let $z_{\max}(t)$ be the maximum of $z(t,x,y)$ at time $t$ over the torus. Then $z_{\max}(t)$ satisfies
\begin{align}\label{equ:evo_zmax}
\frac{d }{d t}(z_{\max}^2(t)-1)\leq -4r^{-1} z_{\max}(t)F^{1/2} (z_{\max}^2(t)-1). 
\end{align}
\end{cor}
\begin{proof} 
Note that at the maximum point of $z$, $\frac{\partial z}{\partial x}=\frac{\partial z}{\partial y}=0$. Together with $F'\geq 0$, \eqref{equ:evo_zmax} follows from \eqref{equ:evo_z}. 
\end{proof}
\begin{proof}[Proof of Lemma~\ref{lem:C1}]
By Lemma \ref{lem:C0} there exists a constant $C_0>0$ satisfying 
\[\vert s(t,x,y)-t\vert < C_0.\] 
We compute
\[\lim_{t\to\infty}t\left(F^{1/2}(s(t,x,y) )-1\right) = 0 > -C_0 = \lim_{t\to\infty}t\left(\frac{t+C_0}{t+2C_0}-1\right).\]
Thus there exists a constant $t_0$ such that for all $t\geq t_0$, 
\[F^{1/2}(s(t,x,y))\geq \frac{t+C_0}{t+2C_0} \geq \frac{s(t,x,y)}{t+2C_0}.\]
It follows that
\[\frac{d }{d t}(z_{\max}^2-1) \leq -4r^{-1}z_{\max}F^{1/2}(z_{\max}^2-1)\leq -4r^{-1}F^{1/2}(z_{\max}^2-1) \leq -\frac{4}{t+2C_0}(z_{\max}^2-1)\]
for all $t\geq t_0$. Integrating the above, it follows that there is a constant $C_1>0$ such that
\[z_{\max}^2(t)-1 \leq C_1(t+1)^{-4}.\]
\end{proof}
\subsection{$C^2$ estimates}\label{sec:5.3} \hfill \\

The goal of this section is to show that the positive part of $H-2$ decays at a rate $t^{-3}\log t$. 
\begin{lemma}\label{lem:C2}
Under the flow \eqref{equ:trueflow}, there exists a constant $C_2>0$ such that 
\begin{equation}\label{equ:asympH}
H-2\leq C_2(1+t)^{-3}\big(\log(1+t)+1 \big).
\end{equation}
\end{lemma}
Let $T_{ij}=h_{ij}-g_{ij}$ and $T=g^{ij}T_{ij}=H-2$. We derive the evolution equation of $T$ in the following.
\begin{lemma}\label{lem:T}
Under the flow \eqref{equ:trueflow},
\begin{equation}\label{equ:T}
\begin{split}
\frac{\partial T}{\partial t}=&\left(-r^{-1}F^{1/2}z^{-1}-2r^{-1} \right)T-r^{-1}|T_{ij}|^2\\
&+ 2r^{-1}F+2^{-1}F'-2r^{-1}-2r^{-1}F^{1/2}z^{-1}-r^{-1}\overline{\textup{Ric}}(\nu,\nu) \\
&-(2r^{-1}   F -  2^{-1}   F')(1-z^{-2})-2^{-1}r^{-4}F^{-2}F' z^{-2}\left( \frac{\partial s}{\partial y} \right)^2.   
\end{split}
\end{equation}
\end{lemma}
\begin{proof}
The evolution equation of $H$ is given by
\begin{align*}
\frac{\partial H}{\partial t}=&\Delta (-r^{-1}) -r^{-1}|A|^2-r^{-1}\overline{\textup{Ric}}(\nu,\nu)\\
=&r^{-2}\Delta r-2r^{-3}|\nabla r|^2-r^{-1}|A|^2-r^{-1}\overline{\textup{Ric}}(\nu,\nu).
\end{align*}
Using \eqref{equ:Lap_r} and \eqref{equ:gradr}, the above equals
\begin{align*}
&r^{-2}\times\bigg( -rF^{1/2}z^{-1}H+  2rF+2^{-1}r^2F'  +  2^{-1}r^2  F'(1-z^{-2})-2^{-1}r^{-2}F^{-2}F' z^{-2}\left( \frac{\partial s}{\partial y} \right)^2   \bigg)\\
&-2r^{-3}\times \bigg( r^2 F(1-z^{-2}) \bigg)-r^{-1}|A|^2-r^{-1}\overline{\textup{Ric}}(\nu,\nu)\\
=&-r^{-1}F^{1/2}z^{-1}H+ 2r^{-1}F+2^{-1} F' -r^{-1}|A|^2-r^{-1}\overline{\textup{Ric}}(\nu,\nu)\\
&+ 2^{-1}   F'(1-z^{-2})-2^{-1}r^{-4}F^{-2}F' z^{-2}\left( \frac{\partial s}{\partial y} \right)^2 -2r^{-1}   F(1-z^{-2}). 
\end{align*}
Since $H=T+2$ and $|A|^2=2+2T+ |T_{ij}|^2,$
\begin{align*}
\frac{\partial T}{\partial t}=&\left(-r^{-1}F^{1/2}z^{-1}-2r^{-1} \right)T-r^{-1}|T_{ij}|^2\\
&+ 2r^{-1}F+2^{-1}F'-2r^{-1}-2r^{-1}F^{1/2}z^{-1}-r^{-1}\overline{\textup{Ric}}(\nu,\nu) \\
&+  2^{-1}   F'(1-z^{-2})-2^{-1}r^{-4}F^{-2}F' z^{-2}\left( \frac{\partial s}{\partial y} \right)^2 -2r^{-1}   F(1-z^{-2}) .
\end{align*}
Then \eqref{equ:T} follows by rearranging terms.
\end{proof}

\begin{lemma}
\label{lemma:Tbound}
Under the flow \eqref{equ:trueflow}, there exists a constant $C_3>0$ such that 
\begin{equation}
\frac{\partial T}{\partial t} - \left(-r^{-1}F^{1/2}z^{-1}-2r^{-1} \right)T + r^{-1}\vert T_{ij}\vert^2 \leq C_3(1+t)^{-4}.
\end{equation}
\end{lemma}
\begin{proof}
By Lemma \ref{lem:T}, we have
\begin{equation}
\label{C4}
\begin{split}
    &\frac{\partial T}{\partial t} - \left(-r^{-1}F^{1/2}z^{-1}-2r^{-1} \right)T + r^{-1}\vert T_{ij}\vert^2 \\
    =&\ 2r^{-1}F+2^{-1}F'-2r^{-1}-2r^{-1}F^{1/2}z^{-1}-r^{-1}\overline{\textup{Ric}}(\nu,\nu) \\
-& (2r^{-1}   F -  2^{-1}   F')(1-z^{-2})-2^{-1}r^{-4}F^2F' z^{-2}\left( \frac{\partial s}{\partial y} \right)^2. 
\end{split}
\end{equation}
Lemma \ref{lem:C0} and Lemma \ref{lem:C1} imply the following asymptotics:
\begin{equation}\label{equ:asympFzr}
\begin{split}
    r^{-1}  = (1+t)^{-1}+O((1+t)^{-2}),\ F = 1 + O((1+t)^{-3}),\\
     F'=O((1+t)^{-4}),\ z  = 1 + O((1+t)^{-4}).
\end{split}
\end{equation}
From Remark \ref{remark:C1} and Lemma~\ref{lemma:ricci}, we have
\begin{align}\label{equ:asymp_ps}
 \frac{\partial s}{\partial y}=O(1),\  \overline{\textup{Ric}}(\nu,\nu)=-2+O((1+t)^{-3}). 
\end{align}
Then we derive
\begin{equation}\label{equ:asymp}
\begin{split}
    2r^{-1}F - 2r^{-1}F^{1/2}z^{-1} &= O((1+t)^{-4}),\\
    -2r^{-1}F^{1/2}z^{-1}-r^{-1}\overline{\textup{Ric}}(\nu,\nu)&= O((1+t)^{-4}),\\
    (2r^{-1}   F -  2^{-1}   F')(1-z^{-2}) &= O((1+t)^{-5}), \\
    2^{-1}r^{-4}F^2F' z^{-2}\left( \frac{\partial s}{\partial y}\right)^2 &= O((1+t)^{-8}).
\end{split}
\end{equation}
Plugging \eqref{equ:asymp} into \eqref{C4} yields the result.
\end{proof}
We are now ready to prove Lemma~\ref{lem:C2}.
\begin{proof}[Proof of Lemma~\ref{lem:C2}]

By Lemma \ref{lemma:Tbound}, we have
\[\frac{\partial T}{\partial t}\leq \left(-r^{-1}F^{1/2}z^{-1}-2r^{-1} \right)T - r^{-1}\vert T_{ij}\vert^2 + C_3(1+t)^{-4}.\]
By \eqref{equ:asympFzr}, it follows that
\[\frac{\partial T}{\partial t} \leq -\left(3(1+t)^{-1} + O((1+t)^{-2})\right)T + C_3(1+t)^{-4}.\]
Define $U=(1+t)^3\max\{T,0\}$. The above implies there exists a constant $C>0$ such that
\[\frac{\partial U}{\partial t}\leq C(1+t)^{-2}U + C_4(1+t)^{-1}.\]
We may rewrite this as
\[ \frac{\partial}{\partial t} \left( U e^{C/(1+t)} \right)\leq C_3(1+t)^{-1}e^{C/(1+t)}\leq C'(1+t)^{-1}.\]
Here $C'=C_3e^C$. Let $U_{\max}(t)$ be the maximum of $U$ at time $t$. By integrating the above inequality, we have
$$U_{\max}(t) e^{C/(1+t)}\leq U_{\max}(0) e^{C }+C'\log(1+t).$$
This implies
\begin{align*}
H-2\leq e^{-C/(1+t)}\left( C'(1+t)^{-3}\log(1+t)+e^{C }(1+t)^{-3}U_{\max}(0)  \right).
\end{align*}
Then \eqref{equ:asympH} follows by setting $C_2=C'+e^C U_{\max}(0)$.
\end{proof}

\subsection{Proof of Proposition~\ref{theorem:theorem5.1}}\label{sec:5.4}
In this section we prove Proposition~\ref{theorem:theorem5.1}. Recall that $\varphi:[0,\infty)\times T^2\to M$ is a family of embeddings satisfying \eqref{equ:trueflow} and that $\Sigma_t$ is the image of $\varphi(t,\cdot)$.
%\begin{lemma}
%\label{lemma:Lap_rbound}
%There exists a constant $C_4>0$ such that 
%\[\Delta r \geq -C_4(1+t)^{-2}\log (1+t).\]
%\end{lemma}
%\begin{proof} 
%This is a direct consequence of \eqref{equ:Lap_r}, \eqref{equ:asympFzr}, \eqref{equ:asymp_ps} and %Lemma \ref{lem:C2}.
%\end{proof}
\begin{lemma}
\label{lemma:Lap_rabs}
Under the flow \eqref{equ:trueflow}, there exists a constant $C_4>0$ such that 
\[\int_{\Sigma_t} \vert\Delta r\vert\,  {dA_t} \leq C_4\big(\log(1+t)+1\big).\]
\end{lemma}
\begin{proof}
Let $(\Delta r)_+ = \frac 12 (\vert \Delta r\vert + \Delta r)$ and $(\Delta r)_- = \frac 12 (\vert \Delta r\vert - \Delta r)$ be the positive and negative parts of $\Delta r$ respectively. By the divergence theorem,
\[0 = \int_{\Sigma_t} \Delta r\, {dA_t} = \int_{\Sigma_t} \left((\Delta r)_+ - (\Delta r)_-\right) {dA_t}.\]
It thus suffices to show that
\[\int_{\Sigma_t} (\Delta r)_-\,  {dA_t}\leq  C \big(\log (1+t)+1\big) .\]
From \eqref{equ:Lap_r}, \eqref{equ:asympFzr}, \eqref{equ:asymp_ps} and Lemma \ref{lem:C2}, we have 
\[(\Delta r)_- \leq C(1+t)^{-2}\big(\log (1+t)+1\big).\]
Similarly, from \eqref{volforms} and \eqref{equ:asympFzr}, we have
\begin{align*}
 {dA_t}\leq C(1+t)^2   dy  dx.
\end{align*}
Combining the two yields the result.
\end{proof}

\begin{lemma}
Under the flow \eqref{equ:trueflow}, we have
\label{lemma:Lap_rlimit}
\[\lim_{t\to\infty}\left(\int_{\Sigma_t} F^{-1/2}z\Delta r\,  {dA_t}\right) = 0.\]
\end{lemma}
\begin{proof}
By the divergence theorem,
\[\int_{\Sigma_t} F^{-1/2}z\Delta r\,  {dA_t} = \int_{\Sigma_t} \left(F^{-1/2}z-1\right)\Delta r\,  {dA_t}.\]
Using \eqref{equ:asympFzr}, we have
\[\int_{\Sigma_t} \left\vert\left(F^{-1/2}z-1\right)\Delta r \right\vert {dA_t} \leq  C   (1+t)^{-3} \int_{\Sigma_t} \left\vert\Delta r \right\vert {dA_t}.\]
The result then follows from Lemma \ref{lemma:Lap_rabs}.
\end{proof}

\begin{lemma}
Under the flow \eqref{equ:trueflow}, we have
\label{lemma:intHlimit}
\[\lim_{t\to\infty}\left(\int_{\Sigma_t} \left(Hr - 2r - 2^{-1}r^{-2}\right)  {dA_t}\right) = 0.\]
\end{lemma}
\begin{proof}
Let us rewrite \eqref{equ:Lap_r} in the form
\begin{equation*}
\begin{split}
r(H-2)=-&F^{-1/2}z\Delta r+  2r(F^{-1/2}z - 1) + 2^{-1}r^2 F^{-1/2}F'z \\
&+ 2^{-1}r^2 F^{-1/2} F'z(1-z^{-2})-2^{-1}r^{-2}F^{-5/2}F' z^{-1}\left( \frac{\partial s}{\partial y} \right)^2  .
\end{split}
\end{equation*}
By Lemma \ref{lem:C1} and Remark \ref{remark:C1}, we have
\begin{equation*}
\begin{split}
    2^{-1}r^2  F^{-1/2}F'z(1-z^{-2}) &= O((1+t)^{-6}), \\
    2^{-1}r^{-2}F^{-5/2}F' z^{-1}\left( \frac{\partial s}{\partial y} \right)^2 &= O((1+t)^{-6}).
\end{split}
\end{equation*}
Therefore
\[r(H-2)=-F^{1/2}z\Delta r+  2r(F^{-1/2}z - 1) + 2^{-1}r^2 F^{-1/2}F'z + O((1+t)^{-6}).\]
Note that 
\begin{align*}
    2r(F^{1/2}z - 1) &= -r^{-2} + o((1+t)^{-2}), \\
    2^{-1}r^2F^{-1/2}F'z &= \frac 32 r^{-2} + o((1+t)^{-2}).
\end{align*}
Thus
\[Hr - 2r - 2^{-1}r^{-2} = -F^{-1/2}z\Delta r + o((1+t)^{-2}).\]
Upon integration, the desired limit follows by using Lemmas \ref{lem:C0} and \ref{lemma:Lap_rlimit}.
\end{proof}
We are now ready to prove Proposition~\ref{theorem:theorem5.1}.
\begin{proof}[Proof of Proposition~\ref{theorem:theorem5.1}]
Recall that
\begin{align*}
Q(\Sigma):=\int_{\Sigma} H r\, dA-6\int_{\Omega} r\, dV.
\end{align*} 
By Lemma \ref{lemma:intHlimit}, it suffices to show that 
\[\lim_{t\to\infty}\left(\int_{\Sigma_t} \left(2r + 2^{-1}r^{-2}\right)\,  {dA_t}-6\int_{\Omega_t} r\,  {dV}\right) = P_xP_y(2r_s^3-2^{-1}).\]
Using \eqref{volforms}, we may rewrite
\begin{align*}
&\int_{\Sigma_t} \left(2r + 2^{-1}r^{-2}\right)\,  {dA_t}-6\int_{\Omega_t} r\,  {dV}-P_xP_y(2r_s^3-2^{-1})\\
 =& \int_0^{P_x}\int_0^{P_y}\left(2r^3F^{1/2}z - 2r^3 + 2^{-1}F^{1/2}z + 2^{-1}\right)  dy  dx.
\end{align*}
From \eqref{def:F}, Lemma~\ref{lem:C0} and Lemma~\ref{lem:C1}, we have
\[2r^3F^{1/2}z - 2r^3 = -1 + o(1), \ 2^{-1}F^{1/2}z + 2^{-1} = 1 + o(1).\]
Upon adding the two and integrating, we obtain the desired result.
\end{proof}

\begin{appendix}
\section{Curvature of the AdS-Melvin space}\label{sec:A}
In this section, we compute the Riemannian curvature tensor of AdS-Melvin space whose metric is given by \eqref{metric}.
\begin{lemma}\label{lem:R}
The Ricci curvature of $\bar{g}$ is given by
\begin{equation}\label{equ:Ric}
\begin{split}
 \overline{R}_{ab}=&\left(-2r^{-2}-\frac{5}{2}r^{-1} F^{-1}  F'-\frac{1}{2} F^{-1}F'' \right)\bar{\nabla}_ar\bar{\nabla}_br  + (-2r^2F-r^3F')\bar{\nabla}_a x\bar{\nabla}_b x  \\
&+\left( -2r^2F-\frac{5}{2}r^3 FF'-\frac{1}{2}r^4 FF'' \right) \bar{\nabla}_ay\bar{\nabla}_by.
\end{split}
\end{equation}
Also, the scalar curvature is given by 
\begin{align}\label{equ:R}
    \bar R &= -6+\frac{2b}{r^4}. 
\end{align}
\end{lemma}

\begin{proof}
We begin by computing the Riemannian curvature tensor of $\bar{g}$ through the Gauss-Codazzi equations. Consider the surface $S_1:=\{y=\textup{constant}\}$. Because $S_1$ is perpendicular to the Killing vector $\frac{\partial}{\partial y}$, it is totally geodesic and has vanishing second fundamental form. By the Gauss equation, the sectional curvature of the plane generated by $\frac{\partial}{\partial r}$ and $\frac{\partial}{\partial x}$ equals the Gauss curvature of $S_1$, which is $-F-2^{-1}r F'$. Therefore, $\overline{R}_{rxrx}=-1-2^{-1}r F^{-1}F'.$ Moreover, by the Codazzi equation, we have $\overline{R}_{yxrx}=\overline{R}_{yrxr}=0.$

Applying the same argument to the surface $S_2=\{x=\textup{constant}\}$, we have $\overline{R}_{ryry}=-F-2 r F'-2^{-1}r^2F''$ and $\overline{R}_{xyry}=0$.

Consider the surface $S_3=\{r=\textup{constant}\}$. From Lemma~\ref{lem:christoffel}, we have that the second fundamental forms of $S_3$ is given by $h_{xx}=r^{2}F^{1/2}$, $h_{yy}=r^2F^{3/2}+2^{-1}r^3F^{1/2}F'$ and $h_{xy}=0$. Notice that the induced metric of $S_3$ is flat. From the Gauss equation, we have $\overline{R}_{yxyx}=-r^4F^2-2^{-1}r^5FF'$. In short, we have
\begin{align*}
\overline{R}_{rxrx}=-1-2^{-1}r F^{-1}F',\ \overline{R}_{ryry}=-F-2 r F'-2^{-1}r^2F'',\\
 \overline{R}_{yxyx}=-r^4F^2-2^{-1}r^5FF',\ \overline{R}_{yxrx}=\overline{R}_{yrxr}=\overline{R}_{xyry}=0.
\end{align*}
Together with \eqref{metric}, we derive
\begin{align*}
\overline{R}_{rr}=&-2r^{-2}-\frac{5}{2}r^{-1}F^{-1}F'-\frac{1}{2} F^{-1}F'',\\
\overline{R}_{xx}=&-2r^{2}F-r^3F',\\
\overline{R}_{yy}=&-2r^{2}F^2-\frac{5}{2}  r^3F'-\frac{1}{2}r^4FF'' ,
\end{align*}
and $\overline{R}_{rx}=\overline{R}_{ry}=\overline{R}_{xy}=0$. This proves \eqref{equ:Ric}. Moreover, we have $R=-6F-6rF'-r^2F''.$ Then \eqref{equ:R} follows by plugging $F=1-r^{-3}-br^{-4}$.
\end{proof}
\begin{lemma}
\label{lemma:ricci}
There exists a constant $C>0$ such that for any unit vector $\nu$ we have
\begin{equation}
\vert\overline{\textup{Ric}}(\nu,\nu)+2\vert \leq Cr^{-3}.
\end{equation}
\end{lemma}
\begin{proof}
From Lemma~\ref{lem:R} and \eqref{def:F}, the eigenvalues of the Ricci curvature are given by 
$$-2+\frac{1}{2r^3}+\frac{2b}{r^4},-2+\frac{1}{2r^3}+\frac{2b}{r^4},\ \textup{and}\ -2-\frac{1}{r^3}-\frac{2b}{r^4}.$$ Then the assertion follows directly.
\end{proof}
\section{Derivation of \eqref{equ:evoarea} and \eqref{equ:evoH}}\label{sec:B}
In this section, we derive \eqref{equ:evoarea} and \eqref{equ:evoH} under the flow \eqref{equ:flow}, which we conveniently recall here as
\begin{align*} 
\frac{\partial  \varphi}{\partial t } = \rho\nu.
\end{align*} 
Here $\varphi:[0,T_0)\times T^2\to M$ is a family of embeddings and $\rho$ is a smooth function on $M$. We write $\Sigma_t$ for the image of $\varphi(t,\cdot)$. We use $g_{ij}$ and $h_{ij}$ to denote the induced metric and the second fundamental form of $\Sigma_t$ respectively. First, we have the evolution equation of the metric
\begin{equation}\label{equ:evometric}
\frac{\partial g_{ij}}{\partial t} = 2\rho h_{ij},\ \frac{\partial g^{ij}}{\partial t} = -2\rho h^{ij}.
\end{equation}
They can be derived as the following.
\begin{proof}[Derivation of \eqref{equ:evometric}]
We have
\begin{align*}
    \frac{\partial g_{ij}}{\partial t} &= \frac{\partial}{\partial t}\bar{g}\left( \frac{\partial\varphi}{\partial  x^i}, \frac{\partial \varphi}{\partial x^j}\right)  =\bar{g}\left( \bar{\nabla}_{\frac{\partial}{\partial t}}\frac{\partial \varphi}{\partial x^i}, \frac{\partial \varphi}{\partial x^j}\right) + \bar{g}\left(\frac{\partial \varphi}{\partial x^i}, \bar{\nabla}_{\frac{\partial}{\partial t}}\frac{\partial \varphi}{\partial x^j}\right).
\end{align*}
\begin{align*}
\bar{g}\left( \bar{\nabla}_{\frac{\partial}{\partial t}}\frac{\partial \varphi}{\partial x^i}, \frac{\partial \varphi}{\partial x^j}\right)=\bar{g}\left( \bar{\nabla}_{\frac{\partial \varphi}{\partial x^i}}\frac{\partial \varphi}{\partial t}, \frac{\partial \varphi}{\partial x^j}\right)=\bar{g}\left( \bar{\nabla}_{\frac{\partial \varphi}{\partial x^i}}(\rho\nu) , \frac{\partial \varphi}{\partial x^j}\right)=\rho h_{ij}.
\end{align*}
Similarly,
\begin{align*}
\bar{g}\left( \frac{\partial \varphi}{\partial x^i}, \bar{\nabla}_{\frac{\partial}{\partial t}}\frac{\partial \varphi}{\partial x^j}\right)=\rho h_{ij}.
\end{align*}
Then we obtain
\begin{align*}
\frac{\partial g_{ij}}{\partial t}=2\rho h_{ij}.
\end{align*}
Also, since
\[0=\frac{\partial}{\partial t}(g_{ik}g^{kj})= g_{ik}\frac{\partial g^{kj}}{\partial t} + g^{kj}\frac{\partial g_{ik}}{\partial t},\]
we obtain
\begin{align*}
    \frac{\partial}{\partial t}g^{ij} &= - g^{ik}g^{j\ell}\frac{\partial g_{k\ell}}{\partial t} = -2\rho g^{ik}g^{j\ell}h_{ij} = -2\rho h^{ij}.
\end{align*}
\end{proof}
Now we are ready to prove \eqref{equ:evoarea}.
\begin{proof}[Derivation of \eqref{equ:evoarea}]
We recall the derivative of a determinant. Let $B(t)=(b_{ij}(t)), 1\leq i,j\leq n$ and $t\in (-\delta,\delta)$ be a smooth family of invertible matrices. Denote by $B^{-1}(t)=(b^{ij}(t))$ the inverse matrix of $B(t)$. Then it holds that
\begin{equation}\label{equ:detB}
\frac{  d}{  dt}\det B = \det B \cdot b^{ji}\frac{  d b_{ij}}{  dt}.
\end{equation} 
In local coordinates $\{x^1,x^2\}$, we have
\begin{align*}
dA_t= \sqrt{\det g}\, dx^1dx^2.
\end{align*}
From \eqref{equ:detB} and \eqref{equ:evometric}, we calculate
\begin{align*}
    \frac{\partial}{\partial t} \sqrt{\det g} &= \frac{1}{2\sqrt{\det g}}\cdot\frac{\partial}{\partial t}(\det g)=\frac{1}{2\sqrt{\det g}} \cdot \det g\cdot g^{ji}\frac{\partial g_{ij}}{\partial t} \\
    &= \frac 12 \sqrt{\det g}\cdot g^{ji} \cdot 2\rho h_{ij}  = \rho H \sqrt{\det g}.
\end{align*}
This implies \eqref{equ:evoarea}.
\end{proof}

Next, we calculate the evolution equation for the unit normal vector $\nu$. Recall that $\nabla$ is the Levi-Civita connection of the induced metric on $\Sigma_t$. The time derivative of $\nu$ is given by
\begin{equation}\label{equ:evonormal}
\frac{\partial\nu}{\partial t} = -\nabla\rho.
\end{equation}
The calculation for \eqref{equ:evonormal} can be found below.
\begin{proof}[Derivation of \eqref{equ:evonormal}]
We compute
\begin{align*}
    \bar{g}\left(\frac{\partial \nu}{\partial t},\frac{\partial \varphi}{\partial x^i} \right) &= -\bar{g}\left(\nu, \frac{\partial^2 \varphi}{\partial t\partial x^i}\right)   = -\bar{g}\left(\nu, \frac{\partial(\rho\nu)}{\partial x^i}\right) \\
    &= -\bar{g}\left(\nu, \rho\frac{\partial\nu}{\partial x^i}+\nu\frac{\partial\rho}{\partial x^i}\right)  = -\frac{\partial\rho}{\partial x^i}.
\end{align*}
Also,
\begin{align*}
\bar{g}\left(\frac{\partial \nu}{\partial t},\frac{\partial \varphi}{\partial x^i} \right)=\frac{1}{2}\frac{\partial t}{\partial t}\bar{g}(\nu,\nu)=0.
\end{align*}
Therefore $\frac{\partial\nu}{\partial t}=-\nabla\rho$.
\end{proof} 
Using \eqref{equ:evonormal}, we derive the evolution equation of the second fundamental form:
\begin{equation}\label{equ:evoh}
 \frac{\partial h_{ij}}{\partial t} = -\nabla_i\nabla_j\rho + \rho \bar{R}\left(\frac{\partial \varphi}{\partial  x^i}, \nu, \nu, \frac{\partial\varphi}{\partial x^j}\right) + \rho h_i^k h_{kj}. 
\end{equation}
\begin{proof}[Derivation of \eqref{equ:evoh}]
%With this, we can calculate the derivative of the second fundamental form:

We have
\begin{align*}
    \frac{\partial h_{ij}}{\partial t} &= \frac{\partial}{\partial t}\bar{g}\left(\bar{\nabla}_{\frac{\partial\varphi }{\partial x^i}}\nu , \frac{\partial\varphi}{\partial t}\right) \\
    &= \bar{g}\left(\bar{\nabla}_{\frac{\partial}{\partial t}}\bar{\nabla}_{\frac{\partial\varphi }{\partial x^i}}\nu, \frac{\partial\varphi}{\partial x^j}\right)+\bar{g}\left(\bar{\nabla}_{\frac{\partial\varphi }{\partial x^i}}\nu, \frac{\partial^2\varphi}{\partial x^j\partial t}\right).
\end{align*}
The first term is
\begin{align*}
    \bar{g}\left(\bar{\nabla}_{\frac{\partial}{\partial t}}\bar{\nabla}_{\frac{\partial\varphi }{\partial x^i}}\nu, \frac{\partial\varphi}{\partial x^j}\right) &= \bar{g}\left(\bar{\nabla}_{\frac{\partial\varphi }{\partial x^i}}\bar{\nabla}_{\frac{\partial}{\partial t}}\nu, \frac{\partial\varphi}{\partial x^j}\right) + \bar{g}\left(\left(\bar{\nabla}_{\frac{\partial}{\partial t}}\bar{\nabla}_{\frac{\partial\varphi }{\partial x^i}}-\bar{\nabla}_{\frac{\partial\varphi }{\partial x^i}}\bar{\nabla}_{\frac{\partial}{\partial t}}\right)\nu, \frac{\partial\varphi}{\partial x^j}\right) \\
    &= \bar{g}\left(\bar{\nabla}_{\frac{\partial\varphi }{\partial x^i}}(-{\nabla} \rho), \frac{\partial\varphi}{\partial x^j}\right) + \bar{R}\left(\frac{\partial\varphi }{\partial x^i}, \frac{\partial}{\partial t}, \nu, \frac{\partial\varphi}{\partial x^j}\right) \\
    &= - {\nabla}_i {\nabla}_j\rho + \bar{R}\left(\frac{\partial\varphi }{\partial x^i}, \frac{\partial}{\partial t}, \nu, \frac{\partial\varphi}{\partial x^j}\right) \\
    &= - {\nabla}_i {\nabla}_j\rho + \rho \bar{R}\left(\frac{\partial\varphi }{\partial x^i}, \nu, \nu, \frac{\partial\varphi}{\partial x^j}\right)
\end{align*}
and the second term is
\begin{align*}
    \bar{g}\left(\bar{\nabla}_{\frac{\partial\varphi }{\partial x^i}}\nu, \frac{\partial^2\varphi}{\partial x^j\partial t}\right) &= \bar{g}\left( h_i^k \frac{\partial\varphi}{\partial x^k}, \frac{\partial(\rho\nu)}{\partial x^j}\right) \\
    &= \bar{g}\left( h_i^k \frac{\partial\varphi}{\partial x^k}, \rho h_j^\ell \frac{\partial\varphi}{\partial x^\ell}\right) \\
    &= \rho h_i^k h_j^\ell g_{k\ell} \\
    &= \rho h_i^k h_{kj}.
\end{align*}
The result follows.
\end{proof}
Finally, we can prove \eqref{equ:evoH} by combining \eqref{equ:evometric} and \eqref{equ:evoh}.
\begin{proof}[Derivation of \eqref{equ:evoH}]
We compute
\begin{align*}
    \frac{\partial H}{\partial t} &= \frac{\partial}{\partial t}(g^{ij}h_{ij}) = \frac{\partial g^{ij}}{\partial t}h_{ij} + \frac{\partial h_{ij}}{\partial t}g^{ij} \\
    &= -2\rho h^{ij}h_{ij} + \left(-\nabla_i\nabla_j\rho + \rho \bar{R}\left(\frac{\partial\varphi}{\partial x^i}, \nu, \nu, \frac{\partial\varphi}{\partial x^j}\right) + \rho h_i^k h_{kj}\right)g^{ij}\\ 
    &= -2\rho\vert A\vert^2 - \Delta \rho - \rho\overline{\text{Ric}}\left(\nu,\nu\right) + \rho\vert A\vert ^2 \\
    &= -\Delta \rho - \rho \overline{\text{Ric}}\left(\nu,\nu\right) - \rho\vert A\vert^2.
\end{align*}
\end{proof}

\end{appendix}

\section*{Acknowledgements}

The authors would like to thank the organizers of the PRIMES-USA program at MIT for providing the opportunity to collaborate on this research project. The authors would also like to think Dr. Tanya Khovanova and Dr. Kent Vashaw for carefully proofreading the paper and offering helpful suggestions.

\end{document}